%% file: S2957_v2.tex
\documentclass[11pt,a4paper,oneside]{amsart}
\input{S2957_top_v2.tex}

\definecolor{applegreen}{rgb}{0.55, 0.71, 0.0}
\definecolor{dgreen}{rgb}{0.0, 0.5, 0.0}

\title[Higman--Thompson and right-angled Coxeter groups]{Higman--Thompson groups and profinite properties of right-angled Coxeter groups}
\author{Samuel M. Corson}
\author{Sam Hughes} 
\author{Philip Möller}
\author{Olga Varghese}
\date{\today}

\address{Samuel M. Corson\\
E. T. S. I. I.\\
Universidad Polit\'{e}cnica de Madrid\\
Jos\'{e} Guti\'{e}rrez Abascal 2\\
28006 Madrid, Spain
}
\email{sammyc973@gmail.com}

\address{Sam Hughes\\
Mathematisches Institut\\
Rheinische Friedrich-Wilhelms-Universit\"at Bonn\\
Endenicher Allee 60\\ 
53115 Bonn, Germany}
\email{sam.hughes.maths@gmail.com; hughes@math.uni-bonn.de}

\address{Philip M\"oller\\ 
Institute of Mathematics\\ 
Heinrich-Heine-University D\"usseldorf\\ 
Universit\"atsstra{\upshape{\ss}}e 1\\
40225 Düsseldorf, Germany}
\email{philip.moeller.2@hhu.de}

\address{Olga Varghese\\
Department of Mathematics\\
University of M\"unster\\ 
Einsteinstra\ss e 62\\
48149 M\"unster, Germany}
\email{olga.varghese@uni-muenster.de}



\begin{document}
	
\pagenumbering{arabic}
	
\begin{abstract}
	We prove that every right-angled Coxeter group (RACG) is profinitely rigid amongst all Coxeter groups.  On the other hand we exhibit RACGs which have infinite profinite genus amongst all finitely generated residually finite groups.  We also establish profinite rigidity results for graph products of finite groups.  Along the way we prove that the Higman--Thompson groups $V_{n}$ are generated by $4$ involutions, generalising a classical result of Higman for Thompson's group $V$.
%
%
%
\end{abstract}

\maketitle

\section{Introduction}
For a  group $G$ we denote by $\mathcal{F}(G)$ the set of isomorphism classes of finite quotients of $G$. Two groups $G$ and $H$ are said to have the same finite quotients if $\mathcal{F}(G) =\mathcal{F}(H)$. A group $G$ is called \emph{profinitely rigid relative to a class of groups $\mathcal{C}$} if $G\in \mathcal{C}$ and for any group $H$ in the class $\mathcal{C}$ whenever $\mathcal{F}(G)=\mathcal{F}(H)$, then $G\cong H$.
By definition, a finitely generated residually finite group $G$ is called \emph{profinitely rigid (in the absolute sense)} if $G$ is profinitely rigid relative to the class consisting of all finitely generated residually finite groups. If $G$ is not profinitely rigid then we say $G$ is \emph{profinitely flexible}.

The \emph{genus} of $G$, denoted by $\mathcal{G}(G)$, is defined as the set of isomorphism classes of finitely generated residually finite groups with the same finite quotients as $G$. A group $G$ is called \emph{almost profinitely rigid} if $\mathcal{G}(G)$ is finite.

The study of profinite rigidity has motivated and been the subject of a plethora of research.  For example, finitely generated nilpotent groups are almost profinitely rigid \cite{Pickel1971} and so are polycyclic groups \cite{GrunewaldPickelSegal1980}.  There are metabelian groups with infinite genus \cite{Pickel1974}.

There has been tremendous progress with regard to $3$-manifold groups; deep work of Bridson--McReynolds--Reid--Spitler shows that there are hyperbolic $3$-manifolds groups which are profinitely rigid in the absolute sense \cite{BridsonMcReynoldsReidSpitler2020} with more examples constructed in \cite{CheethamWest2022}.  Note there are profinitely flexible, albeit with finite genus, torus bundle groups \cite{Stebe1972,Funar2013,Hempel2014}.  

For rigidity within the class of $3$-manifold groups even more is known, we summarise some of the highlights:  the Thurston geometry is detected by the profinite completion \cite{WiltonZalesskii2017}, so are various decompositions \cite{Wilkes2018,Wilkes2018b,WiltonZalesskii2019}, fibring is a profinite invariant \cite{Jaikin2020} (see \cite{HughesKielak2022} for a generalisation), and in a recent breakthrough Yi Liu showed that finite volume hyperbolic $3$-manifolds are almost profinitely rigid amongst $3$-manifold groups \cite{Liu2023}.  An analogous result for generic free-by-cyclic groups was obtained in \cite{HughesKudlinska2023}.

On the other hand many full-size groups are not profinitely rigid.  Platonov--Tavgen' showed that $F_2\times F_2$ is profinitely flexible \cite{PlatonovTavgen1986}.  More examples were given by Bass--Lubotzky \cite{BassLubotzky2000}.  In a different vein Pyber \cite{Pyber2004} showed the genus could be uncountable. Next, Bridson--Grunewald gave examples of profinite flexibility amongst the class of finitely presented groups \cite{BridsonGrunewald2004} and showed that $F_n\times F_n$ for $n\geq 3$ has infinite genus. We generalise this result by showing that $F_2\times F_2$ also has infinite genus (\Cref{freegroupsinfinitegenus}). More recently, Bridson \cite{Bridson2016} showed that the profinite genus amongst finitely presented groups can be infinite.

Despite Coxeter groups being ubiquitous with geometric group theory, the study of their profinite genus has remained elusive.  To the authors knowledge the only work on this topic are the following results:  Bessa--Grunewald--Zalesskii studied the genus of groups within the class of groups that are virtually the direct products of free and surface groups \cite{BessaGrunewaldZalesskii2014}, Kropholler--Wilkes \cite{KrophollerWilkes2016} proved that a right-angled Coxeter group (RACG) is profinitely rigid amongst RACGs, 
Santos Rego--Schwer proved all triangle Coxeter groups are distinguished from each other by their finite quotients \cite{SantosRegoSchwer2022}, 
and the third and fourth author of this article  proved irreducible affine Coxeter groups are profinitely rigid amongst Coxeter groups \cite{MollerVarghese2023}. Finally, a small number of hyperbolic reflection groups are known to be profinitely rigid in the absolute sense \cite{BridsonMcReynoldsReidSpitler2021}. 
To this end we raise and partially answer the following questions.

\begin{question}\label{question}
Are Coxeter groups profinitely rigid in the absolute sense? Are they rigid relative to the class of Coxeter groups?
\end{question}

The second question has particular relevance to the widely studied but unsolved isomorphism problem for Coxeter groups.  See \cite{Muhlherr2006,SantosRegoSchwer2022} 
for surveys on the isomorphism problem. 
Indeed, knowing profinite rigidity amongst the class of Coxeter groups would give an algorithm to determine whether two Coxeter groups are not isomorphic (simply examine their finite quotients).

Our first result shows that graph products of finite groups can be distinguished from each other by their finite quotients. This vastly generalises a result of Kropholler--Wilkes \cite{KrophollerWilkes2016} that RACGs are distinguished from each other by their finite quotients and also provides a completely new proof of their result.  Indeed, their proof relies on the cup product structure of a RACG whereas we proceed using subgroup separability properties and profinite Bass--Serre theory.

\medskip
\begin{duplicate}[\Cref{GraphProductsProfiniteRigidity}]
Let $G_\Gamma$ be a graph product of finite groups. Then, $G_\Gamma$ is profinitely rigid relative to the class of graph products of finite groups.
\end{duplicate}
\medskip

Returning to Coxeter groups, we show that RACGs are profinitely rigid amongst the class of Coxeter groups.  This shows that the relative version of \Cref{question} has a positive answer for RACGs.

\medskip
\begin{duplicate}[\Cref{CoxeterProfiniteRigid}]
Let $W$ be a right-angled Coxeter group.  Then, $W$ is profinitely rigid amongst the class of Coxeter groups.
\end{duplicate}
\medskip

Our next major result shows a strong profinite flexibility statement for direct products of free products.  

\medskip
\begin{duplicate}[\Cref{infinitegenus}]
Let $\ell\geq 4$.  Let $G_1, \ldots, G_{\ell}$ be finitely generated, residually finite groups such that for all $1 \leq j \leq \ell$ the centraliser of a nonabelian subgroup of $G_j$ is trivial.  Assume also that at least four of the $G_j$ have a subgroup of index $2$.  Let $d \geq 2$.  Then, the genus of $$\prod_{i=1}^d\left(\bigast_{j=1}^{\ell}G_j\right)$$ is infinite.
\end{duplicate}
\medskip

As an immediate corollary we obtain that \Cref{question} has a negative answer in the general case.

\begin{corollary}
The genus of the Coxeter group  $(\bigast_{i=1}^4 C_2)\times(\bigast_{i=1}^4 C_2)$ is infinite.
\end{corollary}

We can also show certain graph products, and in particular \emph{irreducible} RACGs, that is RACGs which do not split as a non-trivial direct product, are not profinitely rigid (see \Cref{prof.flex.graph.prod} and \Cref{rem.irr}).

Our proof utilises the strategy due to Platonov--Tavgen' to construct so called \emph{Grothendieck pairs} out of fibre products of quotients to perfect groups with vanishing second integral homology and no finite quotients.  A Grothendieck pair is a subgroup $P\leqslant G$ such that every finite quotient of $P$ is induced by one of $G$ and such that $P$ surjects onto every finite quotient of $G$.  

Due to the fact Coxeter groups are generated by involutions we have to find infinite perfect groups which have vanishing second integral homology, no finite quotients, and are generated by involutions.  To this end we utilise the Higman--Thompson groups $V_{n}$ and work of Kapoudjian \cite{Kapoudjian2002} showing that $H_2(V_{n}; \Z)=0$.  Note that the computation of all of the homology groups of $V_n$ was completed in a recent breakthrough of Szymik--Wahl \cite{SzymikWahl2019}.  Our main new contribution to the theory of $V_{n}$ is the following generalisation of a classical result of Higman about Thompson's group $V$ \cite{Higman1974}.

\medskip
\begin{duplicate}[\Cref{fourinvolutionsgenerate}]  
The group $V_n$, with $n \geq 2$, is generated by four involutions.  
\end{duplicate}
\medskip
It is not hard to see that two involutions do not suffice to generate $V_n$. Hence raise the following question
\begin{question}
Can $V_n$, with $n\geq 2$, be generated by three involutions? Is it $(2,3)$-generated?
\end{question}

We also use the \emph{3/2-generation} of the Higman--Thompson groups $V_n$ due to Donovan--Harper \cite{DonovenHarper2020} to construct more examples of Grothendieck pairs (see \Cref{nicefreeprod} for the result and its preceding paragraph for the definition of 3/2-generation).

We end with a refinement of the first part of \Cref{question} which our methods leave open.

\begin{question}
    Let $W$ be a Coxeter group.  Is $W$ profinitely rigid relative to finitely presented groups?
\end{question}

\subsection*{Structure of the paper}
In \Cref{sec:Cox} we give the necessary background on Coxeter groups and graph products we will need to prove our results.  We also prove that even Coxeter groups are finite subgroup separable (\Cref{FiniteSubgroupsSeparable}); see \Cref{sec:Cox} for the relevant definitions.

In \Cref{sec:Rigidity:background} we give some properties of the profinite completion and recount some needed profinite Bass-Serre theory.  In \Cref{sec:rigidity:Graph} we prove \Cref{GraphProductsProfiniteRigidity}.  In \Cref{sec:Rigidity:RACGs} we prove \Cref{CoxeterProfiniteRigid}.  In \Cref{sec:Rigidity:free} we show that a Coxeter group splitting as a non-trivial free product is detected by the profinite completion.  We use this to prove (\Cref{coxeter.free.rigid}) that a free product of Coxeter groups each profinitely rigid amongst Coxeter groups is again profinitely rigid amongst Coxeter groups.

In \Cref{sec:HT} we describe the necessary background on Higman--Thompson groups $V_n$ and then prove \Cref{fourinvolutionsgenerate}.

In \Cref{sec:Flexibility} we recount the Platonov--Tavgen' construction and combine it with \Cref{fourinvolutionsgenerate} to prove \Cref{infinitegenus}.  We provide several other examples of Grothendieck pairs using different properties of the groups $V_n$ (see \Cref{UniversalCoxeter} and \Cref{nicefreeprod}).  We then go on to show that many graph products of groups are not profinitely rigid in the absolute sense.  See \Cref{prof.flex.graph.prod} and its corollaries.

\subsection*{Acknowledgements}
We want to thank Yuri Santos Rego for useful comments on the previous version of this paper.
The work of the first author is supported by the Basque Government Grant IT1483-22 and Spanish Government Grants PID2019-107444GA-I00 and PID2020-117281GB-I00.  The second author received funding from the European Research Council (ERC) under the European Union's Horizon 2020 research and innovation programme (Grant agreement No. 850930).  The third author is funded by a stipend of the Studienstiftung des deutschen Volkes and  by the Deutsche Forschungsgemeinschaft (DFG, German Research Foundation) under Germany's Excellence Strategy EXC 2044--390685587, Mathematics M\"unster: Dynamics-Geometry-Structure. The fourth author is supported by DFG grant VA 1397/2-2. This work is part of the PhD project of the third author.  Finally, we would like to thank the anonymous referees for a number of very helpful comments and suggestions.  We are especially grateful to the referee who suggested a simplification of the proof of \Cref{fourinvolutionsgenerate} which substantially shortened the argument.

\section{Coxeter groups and graph products of groups}\label{sec:Cox}
In this section we briefly recall the basics on Coxeter groups and graph products of groups that we need for the later sections. More information about Coxeter groups can be found in \cite{Davis2008} and about graph products of groups in \cite{Green1990}.

Let $\Gamma=(V,E)$ denote a finite simplicial graph and $m\colon E\to \mathbb{N}_{\geq 2}$ an edge-labeling. We define the \textit{Coxeter group} $W_\Gamma$ as
$$W_\Gamma:=\left\langle V\mid v^2\text{ for }v\in V, (vw)^{m(\{v,w\})}\ \text{if } \{v,w\}\in E \right\rangle.$$
We say $W_\Gamma$ is  \emph{right-angled} if $m(e)=2$ for every edge $e\in E$. If all edge labels are even, we call $W_\Gamma$ \emph{even}. Given a subset $X\subseteq V$, we define the \textit{special parabolic subgroup} $W_X$ as the subgroup $\langle X\rangle \subseteq W_\Gamma$. By \cite[Page 20]{Bourbaki68}, \cite[Theorem 4.1.6]{Davis2008}, this is well defined since $W_X$ is canonically isomorphic to the Coxeter group defined via the full subgraph spanned by $X$. By definition, any conjugate of a special parabolic subgroup is called \emph{parabolic}. 
Given a family of parabolic subgroups, their intersection is again a parabolic subgroup by \cite{Solomon76} and \cite{Qi2007}. 
Let $W_\Gamma$ be a Coxeter group and $A\subseteq W_\Gamma$ be a subgroup. By \cite[Theorem 1.2]{Qi2007} there exists a unique minimal parabolic subgroup $gW_\Delta g^{-1}$ such that $A\subseteq gW_\Delta g^{-1}$. This parabolic subgroup is called the \emph{parabolic closure} of $A$ and is denoted by $\Pc(A)$. We note that if two subgroups $A$, $B$ are conjugate, then $\Pc(A)$ is conjugate to $\Pc(B)$.

The definition of a graph product of groups is similar. Let $\Gamma=(V,E)$ be a finite simplicial graph and let $f\colon V\to\left\{\text{non-trivial fin. gen. groups}\right\}$ be a vertex-labeling. The \emph{graph product of groups} $G_\Gamma$ is defined as a quotient
$$G_\Gamma:=\left(\bigast_{v\in V} f(v)\right)\bigg/\langle\langle [a,b]\mid a\in f(v), b\in f(w)\text{ if }\left\{v, w\right\}\in E\rangle\rangle$$

The definitions of (special) parabolic subgroups and the parabolic closure of a subgroup are the same as for Coxeter groups. For more details see \cite{Green1990} and \cite{AntolinMinasyan2015}. 

Note that if every vertex group is isomorphic to $C_2$, the graph product of groups $G_\Gamma$ is a right-angled Coxeter group. If every vertex group is isomorphic to $\Z$, we call $G_\Gamma$ a \textit{right-angled Artin group}.

It is known that Coxeter groups 
are linear, see \cite[Page 91]{Bourbaki68}, \cite[Corollary 6.12.11]{Davis2008}.
Hence, Coxeter groups 
are residually finite by \cite{Malcev1940}. It was shown by Green in \cite[Corollary 5.4]{Green1990} that graph products of residually finite groups are residually finite. In particular, graph products of finite groups are residually finite.

Complete subgraphs are important for the study of graph products and  Coxeter groups. In this article we adopt the well known graph-theoretic terminology and call them \textit{cliques}. Note that we also call a parabolic subgroup $gG_\Delta g^{-1}$ a clique if $\Delta$ is a clique.

Finally, we call a graph product of groups $G_\Gamma$ (Coxeter group $W_\Gamma$) \textit{reducible}, if there exists a partition of the vertex set $V=V_1\cup V_2$ such that $\Gamma$ is the join of the induced subgraphs, that is there is an edge (with label $2$) between every pair of vertices $\{v_1,v_2\}$ with $v_1\in V_1$ and $v_2\in V_2$. Otherwise we call the group \textit{irreducible}.

\medskip

Let $G$ be a group and $H, L\leqslant G$ be two non-conjugate subgroups. By definition $H$ is \emph{conjugacy separable} from $L$ if there exists a homomorphism $\varphi\colon G\to F$, with $F$ finite, such that $\varphi(H)$ is not conjugate to $\varphi(L)$.  A group $G$ is said to be \emph{subgroup conjugacy separable} if any two finitely generated non-conjugate subgroups of $G$ are conjugacy separable. For example all limit groups are known to be subgroup conjugacy separable by \cite{ChagasZalesskii2016}.

We call $G$ \emph{finite subgroup conjugacy separable} if any two non-conjugate finite subgroups of $G$ are conjugacy separable.
We call $G$ \emph{conjugacy separable} if any two non-conjugate elements of $G$ are remain non-conjugate in some finite quotient.

Being (subgroup) conjugacy separable should not be confused with being \emph{locally extended residually finite} (LERF) also known as \emph{subgroup separable}.  This latter property states that every finitely generated subgroup $H\leqslant G$ is separable from every element $g\in (G - H)$ in a finite quotient.  Namely, for each such $g$, there exists $\alpha\colon G\to F$, with $F$ finite, such that $\alpha(g)\notin\alpha(H)$.

It was proven by Caprace and Minasyan in \cite[Theorem 1.2]{CapraceMinasyan2013} that an even Coxeter group $W_\Gamma$ is conjugacy separable if $\Gamma$ has no triangles isomorphic to the graph in Figure 1.
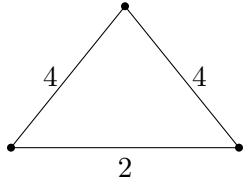
\begin{figure}[h]
\begin{center}
\begin{tikzpicture}[scale=0.75]
    \coordinate (A) at (0,0);
    \coordinate (B) at (-2,-2.5);
    \coordinate (C) at (2,-2.5);

    \draw (A) -- (B) node [midway, left]{$4$};
    \draw (A) -- (C) node [midway, right]{$4$};
    \draw (B) -- (C) node [midway, below]{$2$};

    \fill (A) circle (2pt);
    \fill (B) circle (2pt);
    \fill (C) circle (2pt);
\end{tikzpicture}
\caption{Coxeter graph of type $\widetilde{B}_2$.}
\end{center}
\end{figure}

Further, conjugacy separability of graph products of finite groups was proven by Ferov in \cite[Theorem 1.1]{Ferov2016}.
\medskip

We ask the following question.
\begin{question}
Is every Coxeter group finite subgroup conjugacy separable?
\end{question}

\begin{lemma}
\label{FiniteSubgroupsSeparable}
Let $W_\Gamma$ be a Coxeter group. If $\Gamma$ is even, then $W_\Gamma$ is finite subgroup conjugacy separable.
\end{lemma}




\begin{proof}
It is well known that for an even Coxeter group $W_\Gamma$ and a special parabolic subgroup $W_X$ the homomorphism
$p_X\colon W_\Gamma\twoheadrightarrow W_X$ induced by $p_X(v)=v$ for all $v\in X$ and $p_X(w)=1$
for $w\in(V-X)$ is a well-defined retraction, see
\cite[Proposition~2.1]{Gal2005} for more details on this.

Let $G$ and $H$ be two non-conjugate finite subgroups in $W_\Gamma$. By \cite[Chapter~5 \S4, Exercise~2]{Bourbaki68}, there exist finite special parabolic subgroups $W_I$, $W_J$ and $g,h\in W_\Gamma$ such that $\Pc(G)=gW_Ig^{-1}$ and $\Pc(H)=hW_Jh^{-1}$.
Without loss of generality we replace $G$ and $H$ by $g^{-1}Gg$ and $h^{-1}Hh$ respectively to assume that 
$\Pc(G)=W_I$ and $\Pc(H)=W_J$.  Furthermore, we can assume that $|J|\leq |I|$.

We now show that the retraction $p_I\colon W_\Gamma\twoheadrightarrow W_I$ distinguishes the conjugacy classes of the images of $G$ and $H$.  We have two cases depending on if $J\subseteq I$:

\noindent \textit{Case 1:} If $J\subseteq I$, then $W_J\subseteq W_I$ and $p_I(H)=H\subseteq W_I$ and $p_I(G)=G\subseteq W_I$. By assumption, $G$ and $H$ are not conjugate in $W_\Gamma$, hence $G$ and $H$ are also not conjugate in $W_I$. 

\noindent \textit{Case 2:}  If $J\nsubseteq I$, then $|I\cap J|<|I|$.  By definition, $p_I$ \emph{commutes} with $p_J$ if for all $w\in W_\Gamma$ we have
$$p_I(p_J(w))=p_J(p_I(w)).$$
Note, that all retractions onto special parabolic subgroups in even Coxeter groups commute, see \cite[Page 8]{CapraceMinasyan2013}. Thus, by \cite[Remark 2.4]{CapraceMinasyan2013} the map $p_I\circ p_J=p_J\circ p_I$ is a retraction of $W_\Gamma$ onto $W_{I\cap J}$.

Assume for a contradiction that $p_I(G)\subseteq W_I$ is conjugate to $p_I(H)\subseteq W_I$. Then $\Pc(p_I(G))$ is conjugate to $\Pc(p_I(H))$ in $W_I$.
Now, we consider the structure of these parabolic closures. We have 
$$\Pc(p_I(G))=\Pc(G)=W_I.$$
Since $\Pc(H)=W_J$ we have $p_J(h)=h$ for all $h\in H$. The retractions $p_I$ and $p_J$ commute, thus $p_I\circ p_J$ is the retraction of $W_\Gamma$ onto $W_{I\cap J}$. Therefore we obtain
$$\Pc(p_I(H))=\Pc(p_I\circ p_J(H))=\Pc(p_{I\cap J}(H))\subseteq W_{I\cap J}.$$
Since the cardinality of $I\cap J$ is smaller than the cardinality of $I$ the group $\Pc(p_I(G))$ can not be conjugate to $\Pc(p_I(H))$ (see \cite[Lemma 3.2]{Qi2007}). The contradiction completes the proof. \qedhere
\end{proof}

It follows from \cite[Lemma 3.18]{GenevoisMartin2019} that in a graph product of finite groups the parabolic closure of a finite subgroup is finite. Thus the same proof strategy shows the following lemma.
\begin{lemma}
\label{FiniteSubgroupsSeparableGraphProducts}
Let $G_\Gamma$ be a graph product of finite groups. Then $G_\Gamma$ is finite subgroup conjugacy separable.
\end{lemma}

The following lemma regarding the splitting of graph products of groups as amalgamated products will be useful later on.
\begin{lemma} \emph{\cite[Lemma 3.20]{Green1990}}
\label{GraphProductAmalgam}
    Let $G_\Gamma$ be graph product of finite groups. If there exist two vertices $v,w\in V$ such that $\left\{v,w\right\}\notin E$, then
    $G_\Gamma\cong G_{\st(v)}*_{G_{\lk(v)}}G_{V-\left\{v\right\}}$ where $\lk(v)=\left\{w\in V\mid\left\{v,w\right\}\in E\right\}$ and $\st(v)=\lk(v)\cup\left\{v\right\}$.
\end{lemma}

Denote by $\mathcal{CF}(G)$ the set of conjugacy classes of all finite subgroups in $G$. We define a partial order on $\mathcal{CF}(G)$ as follows: $[A]\leq [B]$ if there exists a $g\in G$ such that $A\subseteq gBg^{-1}$ (see Theorem 4.1 in \cite{Radcliffe2003} for the proof that this is indeed a partial order).

Recall the following definitions. Given a poset $(P,\leq)$ and a subset $A\subseteq P$, a \textit{greatest lower bound of $A$}, denoted by $\bigwedge A$, is an element $x\in P$ such that $x\leq a$ for all $a\in A$ and if $y\leq a$ for all $a\in A$, then $y\leq x$. Further, a \emph{least upper bound} of $A$, denoted by $\bigvee A$, is an element $x\in P$ such that $x\geq a$ for all $a\in A$ and if $y\geq a$ for all $a\in A$, then $y\geq x$. Note that if a greatest lower bound resp. a least upper bound exists, it is unique.

The next result is essentially due to Radcliffe \cite[Theorem~5.4]{Radcliffe2003} although it is not stated in loc. cit. in the form we need.  For completeness we reproduce the argument here.

\begin{thm}[Radcliffe]
\label{GraphProducts}
Let $G_\Gamma$ and $G_\Omega$ denote two graph products of groups with finite directly indecomposable vertex groups. The following are equivalent:
\begin{enumerate}
    \item $G_\Gamma\cong G_\Omega$
    \item $\Gamma\cong \Omega$ (thought of as graphs with vertices labelled by finite directly indecomposable groups)
    \item There exists an order isomorphism $\psi\colon \mathcal{CF}(G_\Gamma)\to \mathcal{CF}(G_\Omega)$ such that for any $[A]$ in $\mathcal{CF}(G_\Gamma)$ and $B\in \psi([A])$ we have $A\cong B$.
\end{enumerate}
\end{thm}

\begin{proof}
It is obvious that (2) implies (1) and that (1) implies (3). We now adapt Radcliffe's arguments to show that (3) implies (2).

A graph $\Lambda$ is a \emph{$T_0$ graph} in the sense of \cite[Page~1080]{Radcliffe2003}, if for every pair of distinct vertices $v, w\in V(\Lambda)$, there exists a maximal clique which contains exactly one of the two. 

\begin{claim*}
    The result holds if $\Gamma$ and $\Omega$ are finite $T_0$ graphs and the vertex groups are arbitrary finite groups (not necessarily directly indecomposable).
\end{claim*}
\begin{claimproof}[Proof of claim]
In \cite[Theorem 4.2, 4.3]{Radcliffe2003}, Radcliffe describes the greatest lower bounds and least upper bounds in the posets $\mathcal{CF}(G_\Gamma)$ and $\mathcal{CF}(G_\Omega)$ for conjugacy classes of parabolic subgroups. We recount this here:  
If $A,B\subseteq V(\Gamma)$ and the groups $\langle A\rangle$ and $\langle B\rangle$ are finite, we obtain 
    $$[\langle A\cap B\rangle ]=[\langle A\rangle]\bigwedge [\langle B\rangle].$$ 
If additionally $\langle A\cup B\rangle$ is finite, then the least upper bound exists and we have 
    $$[\langle A\cup B\rangle ]=[A]\bigvee [B].$$ 

    We now construct a bijection between $V(\Gamma)$ and $V(\Omega)$ using the order isomorphism $\psi$.  Denote by $\mathcal{M}(\Gamma)$ (resp. $\mathcal{M}(\Omega)$) the set of all maximal cliques in $\Gamma$ (resp. $\Omega$). Note that $\psi$ induces a bijection
    $$\psi'\colon \mathcal{M}(\Gamma)\to \mathcal{M}(\Omega)$$
    such that for every $\Delta\in \mathcal{M}(\Gamma)$ we have $\langle \Delta\rangle \cong \langle \psi'(\Delta)\rangle$ (combine \cite[Lemma 4.5]{Green1990} with \cite[Corollary 3.8]{AntolinMinasyan2015}).
    
    Let $v\in V(\Gamma)$, let $\left\{A_1,\ldots, A_n\right\}$ be the set of all maximal cliques of $\Gamma$ such that $v\in V(A_i)$, and let 
    $$\left\{B_1,\ldots, B_m\right\}=\mathcal{M}(\Gamma)\setminus\left\{A_1,\ldots, A_n\right\}.$$ 
    Since $\Gamma$ is $T_0$ we have
    $\left\{v\right\}=\left(\bigcap_{i=1}^n A_i\right)\setminus\left(\bigcup_{j=1}^m B_j\right).$
    Let $A=\bigcap_{i=1}^n A_i$. 
    Then,
    $$\left[G_A\right]=\left[G_{\bigcap_{i=1}^n A_i}\right]
    =\bigwedge_{i=1}^n\left[G_{A_i}\right]
    \mapsto\bigwedge_{i=1}^n \psi\left[G_{A_i}\right]=\bigwedge_{i=1}^n \left[G_{\psi'(A_i)}\right]
    =\left[G_{\bigcap_{i=1}^n \psi'(A_i)}\right].$$

    We define $\phi(A)=\bigcap_{i=1}^n \psi'(A_i)$.  Now, $A\setminus\{v\}=\bigcup_{j=1}^m(A\cap B_j)$ 
    and
    \begin{multline}
    [G_{A\setminus\{v\}}]=[G_{\bigcup_{j=1}^m(A\cap B_j)}]=
    \bigvee_{j=1}^m [G_{A\cap B_j}]
    =\bigvee_{j=1}^m \left([G_A]\wedge [G_{B_j}]\right) \\\nonumber
    \mapsto\bigvee_{j=1}^m \left(\psi[G_{A}] \wedge \psi[G_{B_j}]\right) 
    = \bigvee_{j=1}^m \left([G_{\phi(A)}]\wedge [G_{\psi' (B_j)}] \right) 
    = [G_{\bigcup_{j=1}^m\phi(A)\cap\psi'(B_j)}].
    \end{multline}
    
    We define $\phi(A\setminus\{v\})= \bigcup_{j=1}^m\phi(A)\cap\psi'(B_j)$.  Now,
    $$ G_v \cong G_A / G_{A\setminus\{v\}}\cong G_{\phi(A)}/G_{\phi(A\setminus \{v\})}\cong G_{\bigcap_{i=1}^n \psi'(A_i) \setminus\left(\bigcup_{j=1}^m\phi(A)\cap\psi'(B_j)\right)}.$$
    In particular, $X_v=\bigcap_{i=1}^n \psi'(A_i) \setminus\left(\bigcup_{j=1}^m\phi(A)\cap\psi'(B_j)\right)$ is non-empty.  Moreover, by the $T_0$ condition $X_v$ contains at most one element.  Indeed, if not, then two distinct elements in $X_v$ would be contained in the same set of maximal cliques of $\Omega$, contradicting the $T_0$ condition.  Thus, $X_v$ has a unique element which is $\phi(v)$. 
    Repeating the argument with the inverses of $\psi$ and $\psi'$ yields that $\phi_{|V(\Gamma)}\colon V(\Gamma)\to V(\Omega)$ is our desired bijection.  Note that by construction $G_v\cong G_{\phi(v)}$.

    Next, we show that $\phi_{|V(\Gamma)}$, which we for simplicity denote by $\phi$, extends to a graph isomorphism $\phi\colon\Gamma\to\Omega$.  Let $e=\{u,v\}$ denote an edge of $\Gamma$.  There exists a maximal clique $A$ containing $e$ and so $\phi(e)=\{\phi(u),\phi(v)\}$ is contained in $\psi'(A)$, a maximal clique of $\Omega$, where $G_A\cong G_{\psi'(A)}$.  But, then $\phi(e)$ is an edge of $\Omega$ as required.  Whence, $\phi$ is a graph isomorphism.  This completes the proof of the claim.
    \end{claimproof}

    Suppose $\Lambda$ is any finite graph.  Following Radcliffe, we say two vertices are \emph{equivalent} if they cannot be distinguished by the cliques of $\Gamma$. The set of equivalence classes $\{\overline v\}$ is a modular partition and the quotient graph, the \emph{$T_0$-quotient} $\overline \Lambda$ of $\Lambda$, satisfies the $T_0$ condition.

    We now treat the case that $\Gamma$ and $\Omega$ are arbitrary finite graphs with directly indecomposable vertex groups (following \cite[Theorem~5.4]{Radcliffe2003}).  Let $\overline \Gamma$ and $\overline \Omega$ be the respective $T_0$-quotients of $\Gamma$ and $\Omega$.  If $\overline v\in V(\overline \Gamma)$, then $G_{\overline v}=\prod_{u\in \overline v} G_u$.  By the earlier case we proved above, we have a graph isomorphism $\phi\colon \overline\Gamma \to \overline\Omega$ such that $G_{\overline v}\cong G_{\phi(\overline{v})}$.  Now, every finite group has a unique factorisation into a product of directly indecomposable groups (up to isomorphism and ordering of factors).  Thus, we obtain bijections $\phi_{\overline v}\colon \overline v\to \phi(\overline v)$ such that $G_u\cong G_{\phi_{\overline v}(u)}$ for all $u\in \overline v$.  Let $\theta\colon V(\Gamma) \to V(\Omega)$ be the union of the maps $\phi_{\overline v}$.  Then, $\theta$ is easily seen to be a bijection $V(\Gamma)\to V(\Omega)$ and to extend to a graph isomorphism $\theta\colon\Gamma\to \Lambda$ satisfying $G_v\cong G_{\theta(v)}$. 
\end{proof}

\section{Profinite rigidity}\label{sec:Rigidity}

\subsection{Background on profinite completions}\label{sec:Rigidity:background}

Let $G$ be a group and $\mathcal{N}$ be the set of all finite index normal subgroups of $G$. We equip each $G/N$, $N\in\mathcal{N}$ with the discrete topology and endow $\prod_{N\in\mathcal{N}} G/N$ with the product topology. We define a map $i\colon G\to \prod_{N\in\mathcal{N}} G/N$ by $g\mapsto (gN)_{N\in\mathcal{N}}$. Note, that if $G$ is residually finite, then $i$ is injective. The \emph{profinite completion} of $G$, denoted by $\widehat{G}$, is defined as $\widehat{G}:=\overline{i(G)}$.  Equivalently, $\widehat{G}$ may be constructed as the inverse limit $\varprojlim_{N\in\caln}G/N$.

The homomorphism $i$ has the following universal property:  Let $\mathbf{H}$ be a profinite group and $\phi\colon G\to\mathbf{H}$ be a continuous homomorphism.  Then, there exists a unique continuous homomorphism $\widehat\phi\colon\widehat{G}\to\mathbf{H}$ making the following diagram commute
\[\begin{tikzcd}
    G \arrow[r,"i"] \arrow[rd,"\phi"] & \widehat{G} \arrow[d,"\widehat{\phi}"]\\
    & \mathbf{H}.
\end{tikzcd}\]

The next theorem shows that the set $\mathcal{F}(G)$ of isomorphism classes of finite quotients of a finitely generated residually finite group $G$ encodes the same information as $\widehat{G}$.

\begin{thm} \emph{\cite{Dixon1982}}
\label{Dixon}
Let $G$ and $H$ be finitely generated residually finite groups. Then $\mathcal{F}(G)=\mathcal{F}(H)$ if and only if $\widehat{G}\cong\widehat{H}$.
\end{thm}

Note that by the work of Nikolov--Segal \cite{NikolovSegal2007I,NikolovSegal2007II} we have that $\widehat{G}$ is isomorphic to $\widehat{H}$ as a topological group if and only if $\widehat{G}$ is isomorphic to $\widehat{H}$ as an abstract group.

For profinite groups $\mathbf G_1$ and $\mathbf G_2$ with a common closed subgroup $\mathbf H$, we denote the pushout  $\mathbf G_1$ and $\mathbf G_2$ over $\mathbf{H}$ by $\mathbf G=\mathbf{G}_1\coprod_\mathbf{H} \mathbf{G}_2$.  If the natural maps from $\mathbf G_1$ and $\mathbf G_2$ to $\mathbf G$ are embeddings then we call $\mathbf G$ the \emph{profinite amalgamated product} of $\mathbf G_1$ and $\mathbf G_2$ along $\mathbf H$.

\begin{thm}\emph{\cite[\S 5.6]{MelnikovZalesskii1989}}
\label{AmalgamFiniteSubgroups}
Let $\mathbf G_1\coprod_\mathbf{H} \mathbf G_2$ be a profinite amalgamed product. If $F\subseteq \mathbf G_1\coprod_\mathbf{H} \mathbf G_2$ is a finite subgroup, then $F$ is contained in a conjugate of $\mathbf G_1$ or $\mathbf G_2$.
\end{thm}

\begin{lemma}
\label{AmalgamProfCompletion}
Let $G\cong A\ast_C B$ be a finitely generated residually finite group. If
$A$, $B$ and $C$ are retracts of $G$,
then $\widehat{G}\cong\widehat{A}\coprod_{\widehat{C}}\widehat{B}$.
\end{lemma}
\begin{proof}
    By \cite[Propositon~6.5.3]{RibesBook} $\widehat{G}\cong\overline{A}\coprod_{\overline{C}}\overline{B}$ where $\overline{H}$ denotes the closure of the image of $H$ in $\widehat G$.  Since $A$, $B$, and $C$ are retracts of $G$ each of them has the full profinite topology induced by $G$.  That is, every finite index subgroup of $A$ (resp. $B$, $C$) is a closed subgroup of $G$ (with respect to the profinite topology on $G$).  Indeed, (virtual) retracts are closed in the profinite topology \cite[Lemma~2.2]{Minasyan2021} and finite index subgroups of (virtual) retracts are virtual retracts \cite[Lemma~3.2(iv)]{Minasyan2021}.  Hence, $\overline{A}\cong\widehat A$ (and similarly for $B$ and $C$).
\end{proof}

\subsection{Profinite rigidity and graph products}\label{sec:rigidity:Graph}

The following lemma is straightforward. 

\begin{lemma}
\label{FSSInjective}
Let $G$ be a finitely generated residually finite group. If $G$ is finite subgroup conjugacy separable, then the canonical inclusion $G\hookrightarrow \widehat{G}$ induces an injective map $\colon\mathcal{CF}(G)\to\mathcal{CF}(\widehat{G})$.
\end{lemma}

\begin{prop}
\label{EvenCliqueSubgroupsFinite1}
Let $G_\Gamma$ be a graph product of finite directly indecomposable groups.
Then, the canonical inclusion $G_\Gamma\hookrightarrow \widehat{G_\Gamma}$ induces an order isomorphism $\varphi\colon\mathcal{CF}(G_\Gamma)\to\mathcal{CF}(\widehat{G_\Gamma})$ such that for any $[A]\in\mathcal{CF}(G_\Gamma)$ and $B\in \varphi([A])$ we have $A\cong B$.
\end{prop}
\begin{proof}
By \Cref{FiniteSubgroupsSeparableGraphProducts} the graph product $G_\Gamma$ is finite subgroup conjugacy separable. Since $G_\Gamma$ is a finitely generated residually finite group, the injectivity of $\varphi$ follows by \Cref{FSSInjective}.  It remains to show surjectivity and that $\varphi$ is order preserving.

If $\Gamma$ is a clique, then $G_\Gamma$ is finite and the conclusion of the proposition follows immediately since $\widehat{G_\Gamma}=G_\Gamma$. Thus, we may assume that $\Gamma$ is not a clique. By \Cref{GraphProductAmalgam} the group $G_\Gamma$ is an amalgamated product of special subgroups $G_\Gamma=A*_C B$. Special parabolic subgroups of $G_\Gamma$ are retracts of $G_\Gamma$, therefore we can apply \Cref{AmalgamProfCompletion} to obtain $\widehat{G_\Gamma}=\widehat{A}\coprod_{\widehat{C}} \widehat{B}$.

Let $[F]\in\mathcal{CF}(\widehat{G_\Gamma})$. Then $F$ is contained in a conjugate of $\widehat{A}$ or $\widehat{B}$ by \Cref{AmalgamFiniteSubgroups}. If $A$ is not a clique, then we decompose $A$ again into an amalgamated product. Repeating this process finitely many times we obtain that $F$ is contained in a conjugate of $\widehat{A'}$ where $A'$ is a clique and hence a finite subgroup. Thus, $\widehat{A'}=A'$ and $F\subseteq gA'g^{-1}$. In particular, there exists a finite subgroup $A^{''}\subseteq G_\Gamma$ such that $F=gA^{''}g^{-1}$. Hence, $\varphi([A^{''}])=[F]$, which shows the surjectivity of $\varphi$. Clearly, by construction $\varphi$ is order preserving.
\end{proof}

We now prove our first result from the introduction.

\begin{thm}
\label{GraphProductsProfiniteRigidity}
    Let $G_\Gamma$ and $G_\Omega$ be graph products of finite groups. Then $G_\Gamma\cong G_\Omega$ if and only if $\widehat{G_\Gamma}\cong\widehat{G_\Omega}$.
\end{thm}
\begin{proof}
If $G_\Gamma\cong G_\Omega$, then $\widehat{G_\Gamma}\cong\widehat{G_\Omega}$ by \Cref{Dixon}. 

We note that if a vertex group splits as a direct product, we can replace the corresponding vertex with a clique corresponding to the direct decomposition to obtain an isomorphic graph product of finite directly indecomposable vertex groups. Thus, we can assume without loss of generality that all vertex groups are directly indecomposable.

If $\widehat{G_\Gamma}\cong\widehat{G_\Omega}$, then by \Cref{EvenCliqueSubgroupsFinite1} we get order isomorphisms
$$\mathcal{CF}(G_\Gamma)\to\mathcal{CF}(\widehat{G_\Gamma})\to\mathcal{CF}(\widehat{G_\Omega})\to\mathcal{CF}(G_\Omega).$$
Finally, by \Cref{GraphProducts} we obtain $G_\Gamma\cong G_\Omega$.
\end{proof}

\subsection{Profinite rigidity and RACGs}\label{sec:Rigidity:RACGs}

\begin{lemma}
\label{EvenCliqueSubgroupsFinite}
    Let $W_\Gamma$ and $W_\Omega$ be Coxeter groups. Assume that $W_\Gamma$ is a right-angled Coxeter group. If $\widehat{W_\Gamma}\cong\widehat{W_\Omega}$, then $W_\Omega$ is also a right-angled Coxeter group.
\end{lemma}
\begin{proof}
    Maximal finite subgroups of $W_\Gamma$ are isomorphic to direct products of cyclic groups of order $2$.  By \Cref{EvenCliqueSubgroupsFinite1} we see that maximal finite subgroups in $\widehat{W_\Gamma}$ are isomorphic to direct products of cyclic groups of order $2$. Since Coxeter groups are residually finite we have $W_\Omega$ embeds into $\widehat{W_\Gamma}$.  Thus, finite subgroups in $W_\Omega$ are isomorphic to direct products of cyclic groups of order $2$.  It follows that the Coxeter group $W_\Omega$ is right-angled.
\end{proof}

We now prove our second result from the introduction.  As explained earlier the result generalises \cite[Theorem~6]{KrophollerWilkes2016} from RACGs to the class of all Coxeter groups.

\begin{thm}
\label{CoxeterProfiniteRigid}
Let $W_\Gamma$ be a right-angled Coxeter group. Then, $W_\Gamma$ is profinitely
rigid relative to the class of Coxeter groups.
\end{thm}
\begin{proof}
Let $W_\Gamma$ be a RACG and let $W$ be a Coxeter group such that $\widehat{W_\Gamma}\cong\widehat{W}$. By \Cref{EvenCliqueSubgroupsFinite} it follows that $W$ is a RACG and so by \Cref{GraphProductsProfiniteRigidity} we obtain $W_\Gamma\cong W$. 
\end{proof}

\subsection{Coxeter groups and profinite free products}\label{sec:Rigidity:free}

\begin{prop}
\label{coxeter.free.products}
Let $W_\Gamma$ be a Coxeter group. Then $\widehat{W_\Gamma}$ splits as a non-trivial profinite free product if and only if $\Gamma$ is disconnected.
\end{prop}
\begin{proof}
Compare to the proof of \cite[Theorem 11]{KrophollerWilkes2016}. 
If $\Gamma$ is disconnected, then $W_\Gamma\cong \bigast_{i=1}^n W_{\Gamma_i}$ where $\Gamma_1,\ldots, \Gamma_n$ are the connected components of $\Gamma$. Thus, $\widehat{W_\Gamma}$ is the free profinite product of $\widehat{W_{\Gamma_1}},\ldots,\widehat{W_{\Gamma_n}}$ by \Cref{AmalgamProfCompletion}.

Assume now that $\Gamma$ is connected and $\widehat{W_\Gamma}$ splits as a profinite free product $G\coprod H$. This splitting induces an action of $\widehat{W_\Gamma}$ on the profinite tree $T$ 
associated to this splitting, in particular the edge stabilisers are trivial and the vertex stabilisers are conjugates of $G$ or $H$, see \cite[Proposition 3.8]{MelnikovZalesskii1989}. Let $\left\{v,w\right\}\in E(\Gamma)$ and recall that $\langle v,w\rangle$ is a finite (dihedral) group. Then, the subgroups $\langle v\rangle $, $\langle w\rangle $ and $\langle v,w\rangle$ fix the same vertex of $T$. Since $\Gamma$ is connected, it follows that $W_\Gamma$ fixes a vertex of $T$. Hence, $\widehat{W_\Gamma}$ fixes a vertex of $T$ and therefore at least one of $G$ and $H$ is trivial. 
\end{proof}

\begin{thm}\label{coxeter.free.rigid}
    Let $W_\Gamma$ be a Coxeter group and suppose $W_\Gamma\cong \bigast_{i=1}^n W_{\Gamma_i}$ where $\Gamma_1,\ldots,\Gamma_n$ are the connected components of $\Gamma$.  If each $W_{\Gamma_i}$ is profinitely rigid amongst Coxeter groups, then so is $W_\Gamma$. 
\end{thm}
\begin{proof}
Let $W_\Omega$ be a Coxeter group. We have $W_\Omega=\bigast_{j=1}^m W_{\Omega_i}$ where $\Omega_1, \ldots, \Omega_m$ are the connected components of $\Omega$. Assume that $\widehat{W_\Gamma}\cong\widehat{W_\Omega}$. By \Cref{AmalgamProfCompletion} we have $\coprod_{i=1}^n \widehat{W_{\Gamma_i}}\cong\coprod_{j=1}^m\widehat{W_{\Omega_j}}$.

Let $$\varphi\colon \coprod_{i=1}^n \widehat{W_{\Gamma_i}}\to\coprod_{j=1}^m\widehat{W_{\Omega_j}} $$
be an isomorphism. 

We consider the action of $\coprod_{j=1}^m\widehat{W_{\Omega_j}}$ on the associated profinite tree $T$. Note that the edge stabilisers are trivial and the vertex stabilizers are conjugates of $\widehat{W_{\Omega_j}}$, $j=1,\ldots, m$, see \cite[Proposition 3.8]{MelnikovZalesskii1989}. We denote this action by $\psi$. 

The proof of \Cref{coxeter.free.products} shows that  the fixed point set of $\psi\circ\varphi(\widehat{W_{\Gamma_i}})$ is non-empty. Thus $\varphi(\widehat{W_{\Gamma_i}})$ is contained in a conjugate of $\widehat{W_{\Omega_j}}$. Similarly, $\varphi^{-1}(\widehat{W_{\Omega_j}})$ is contained in a conjugate of $\widehat{W_{\Gamma_k}}$. By \cite[Proposition 4.5]{deBessaPortoZalesskii2022} it follows that $n=m$ and that there exists $\pi\in{\rm Sym}(n)$ such that $\widehat{W_{\Gamma_k}}\cong\widehat{W_{\Omega_{\pi(k)}}}$. By assumption, for $k=1,\ldots, n$, the Coxeter group $W_{\Gamma_{k}}$ is profinitely rigid amongst Coxeter groups, thus $W_{\Gamma_{k}}\cong W_{\Omega_{\pi(k)}}$. Hence \[W_\Gamma=\bigast_{i=1}^n W_{\Gamma_i}\cong \bigast_{j=1}^n W_{\Omega_j}=W_\Omega.\qedhere\]
\end{proof}

\section{A diversion: Higman-Thompson groups \texorpdfstring{$V_n$}{Vn}}\label{sec:HT}

We first give a brief description of Thompson's group $V$ and the Higman--Thompson groups $V_n$ (the interested reader can find more details in \cite[\textsection 6]{CannonFloydParry1996}, see also \cite{Higman1974}).  Let $2 = \{0, 1\}$, let $2^*$ denote the set of finite sequences of $0$s and $1$s, and $\epsilon$ denote the empty sequence.  For $\alpha, \beta \in 2^*$ we let $\alpha\beta$ denote the concatenation of $\alpha$ with $\beta$.  Write $\alpha \preceq \beta$ if $\alpha$ is a prefix of $\beta$ (i.e. there exists $\gamma \in 2^*$ such that $\beta = \alpha\gamma$), so for example, $101 \preceq 10100$.  We say two elements are \emph{incomparable} if neither is a prefix of the other.  Write $\alpha2^*$ for the set of elements of $2^*$ which have $\alpha$ as a prefix.

The set $2^*$ together with the relation $\preceq$ is the complete binary rooted tree, with $\epsilon$ being the root.  We will call a subset $\tau \subseteq 2^*$ a \emph{rooted subtree} if $\tau$ is closed under taking prefixes.  An element $\alpha$ in rooted a subtree $\tau$ is a \emph{leaf} if there does not exist $\beta \in \tau \setminus \{\alpha\}$ with $\alpha \preceq \beta$.  Let $\mathfrak{T}_2$ denote the set of all rooted subtrees $\tau$ of $2^*$ such that

\begin{enumerate}

\item $|\tau| < \infty$;

\item $\epsilon \in \tau$ (so $\tau \neq \emptyset$);

\item $\alpha 0 \in \tau$ if and only if $\alpha 1 \in \tau$.

\end{enumerate}
\noindent For example, $\{\epsilon, 0, 1, 00, 01\}$ is in $\mathfrak{T}_2$ and each of $00$, $01$, and $1$ is a leaf.  Note that an element $\tau$ of $\mathfrak{T}_2$ is completely determined by its set $\operatorname{Lea}(\tau)$ of leaves.  Supposing that $\tau_0, \tau_1 \in \mathfrak{T}_2$ have the same number of leaves, a bijection $b: \operatorname{Lea}(\tau_0) \rightarrow \operatorname{Lea}(\tau_1)$ extends to a bijection $$f_b: \bigcup_{\alpha \in \operatorname{Lea}(\tau_0)} \alpha2^* \rightarrow \bigcup_{\beta \in \operatorname{Lea}(\tau_1)} \beta2^* \quad\text{ by }\quad f_b(\alpha\gamma) = b(\alpha)\gamma.$$  Clearly, $2^* \setminus \bigcup_{\alpha \in \operatorname{Lea}(\tau)} \alpha2^*  = \tau \setminus \operatorname{Lea}(\tau)$ for each $\tau \in \mathfrak{T}_2$, and so the bijections $f_b$ and $f_b^{-1}$ are each defined on a cofinite subset of $2^*$.

If $f: X \rightarrow Y$ and $g: Z \rightarrow U$ are bijections with $X, Y, Z, U$ cofinite subsets of $2^*$, we consider $f$ and $g$ to be equivalent provided they are equal on a cofinite set in $2^*$, and write $[f]$ for the equivalence class of such a bijection $f$.  The set of all such equivalence classes is easily seen to form a group, with binary operation being given by composing appropriately restricted domains and with inverses given by $[f]^{-1} = [f^{-1}]$.  Thompson's group $V$ is the subgroup consisting of those $[f_b]$ where $b: \operatorname{Lea}(\tau_0) \rightarrow \operatorname{Lea}(\tau_1)$ is a bijection for some suitable members of $\tau_0,\tau_1\in\mathfrak{T}_2$; one easily checks that $[f_b][f_c]$ is indeed equal to $[f_d]$ for some bijection $d$.  Moreover, each element $[f_b] \in V$ has a unique canonical representative $f_b \in [f_b]$ such that the trees associated with the bijection $b$ are minimal (under $\subseteq$).  Thus, we will favour the less cumbersome notation $b$ instead of $[f_b]$ when writing elements of $V$.  \emph{We use the convention in this section that actions are on the right, so the product $b \circ c$ is understood to first apply $b$ and then to apply $c$.}

One can generalize this group by changing the parameter $2$ used in describing $V$.  More precisely, for $n \geq 2$ we write $n = \{0, 1, \ldots, n - 1\}$, take $n^*$ to be the set of finite sequences of elements in $\{0, 1, \ldots, n - 1\}$, define $\alpha \preceq \beta$, and (rooted) subtrees of $n^*$ in the analogous way.  The set $\mathfrak{T}_n$ is the set of those rooted trees $\tau$ which are finite, nonempty and such that $\alpha 0 \in \tau$ if and only if $\alpha 1 \in \tau$ if and only if $\ldots$ if and only if $\alpha(n-1) \in \tau$.  Leaves of a $\tau \in \mathfrak{T}_n$ are defined as before, and finally the group $V_n$ is defined analogously as those equivalence classes of bijections between cofinite subsets of the tree $n^*$ which are produced by a bijection between the leaves of a pair of trees in $\mathfrak{T}_n$.  Particularly $V = V_2$ under this notation and $V_n = G_{n, 1}$ under the notation of Higman \cite{Higman1974} (we shall not vary the parameter $1$ as is done in defining general Higman-Thompson groups).

We point out that when the trees associated with the domain and codomain of a canonical $b \in V_n$ are equal, say $\tau$, then $b$ is an element of the subgroup $\operatorname{Sym}(\operatorname{Lea}(\tau))$ of $V_n$.  In this situation we use the standard cyclic notation for elements of the symmetric group.  Also, if $\alpha$ and $\beta$ are incomparable we use $(\alpha \hspace{.1cm} \beta)$ to denote the transposition of $\alpha$ and $\beta$ in $\operatorname{Sym}(\tau)$ where $\tau$ is the minimal element in $\mathfrak{T}_n$ with $\alpha, \beta \in \operatorname{Lea}(\tau)$.  It is also worth pointing out that an element in $\operatorname{Sym}(\operatorname{Lea}(\tau))$ might not be the canonical representative (for example, if $b: \operatorname{Lea}(\tau) \rightarrow \operatorname{Lea}(\tau)$, $\operatorname{Lea}(\tau) \supseteq \{0i\}_{i < n}$ and $b$ fixes all elements of $\{0i\}_{i < n}$ then $b$ can be represented by a bijection on the leaves of the strictly smaller tree $\tau \setminus \{0i\}_{i < n}$.

When $n \geq 2$ is even the group $V_n$ is simple.  When $n \geq 2$ is odd the group $V_n$ has a simple subgroup of index $2$ which we will describe \cite[\textsection 5]{Higman1974} (see \cite[Theorem 5.4]{Higman1974} for simplicity in both cases).  Fixing $n \geq 3$ odd, an element of $V_n$ given by bijection $b: \operatorname{Lea}(\tau_0) \rightarrow \operatorname{Lea}(\tau_1)$ is \emph{even} (respectively \emph{odd}) if the number of pairs $\alpha_1, \alpha_2 \in \operatorname{Lea}(\tau_0)$ such that $\alpha_1$ is lexicographically below $\alpha_2$ and $(\alpha_1)b$ is lexicographically above $(\alpha_2)b$ is even (resp. odd).  When $\tau_0 = \tau_1$ this corresponds to the standard terminology for the symmetric group $\operatorname{Sym}(\operatorname{Lea}(\tau_0))$.  Using the oddness of $n$ one can show that the even elements form a subgroup, which we will denote $V_n^+$.

We adopt the convention that a transposition $(1\ 1)$ is trivial, so $(1\ 1)(2\ 3)=(2\ 3)$.

Graham Higman has pointed out that the group $V$ can be generated by four involutions \cite[page 49]{Higman1974}.  We generalise this as follows.

\begin{thm}\label{fourinvolutionsgenerate}  The group $V_n$, with $n \geq 2$, is generated by four involutions.  
\end{thm}

\begin{proof}

As Higman has already considered the case $n = 2$, we fix $n \geq 3$.  The group $V_n$ has a non-trivial involution, for example $(0\hspace{.1cm} 1)$.  If $n$ is even, the group $V_n$ is simple and the set of non-trivial involutions in $V_n$ is nonempty and closed under conjugation so we see that $V_n$ is generated by its set of involutions.  If $n$ is odd, the involution described above is in $V_n \setminus V_n^+$ (it is a transposition) and it is easy to see that there are involutions in $V_n^+$ (say, a product of two disjoint transpositions of the leaves of a tree).  As $V_n^+$ is simple we have $V_n^+$ is generated by the set of involutions in $V_n^+$, and so $V_n$ is generated by its set of involutions since $V_n^+$ has index $2$ in $V_n$.  In either case, $V_n$ is generated by the set of involutions.

Next, if $b: \operatorname{Lea}(\tau_0) \rightarrow \operatorname{Lea}(\tau_1)$ is a canonical representative of an involution in $V_n$ then of course $b^{-1}: \operatorname{Lea}(\tau_1) \rightarrow \operatorname{Lea}(\tau_0)$ is the canonical representative of the inverse, so $b = b^{-1}$ and $\tau_0 = \tau_1$.  Therefore $b$ is an involution in $\operatorname{Sym}(\operatorname{Lea}(\tau_0))$ and is a product of commuting transpositions.  Thus $V_n$ is generated by the set of all transpositions of subsets of the form $\operatorname{Lea}(\tau)$, with $\tau \in \mathfrak{T}_n$ (a presentation for $V$ using such generators is given in \cite[Theorem 1.1]{BleakQuick2017}).

Let $\overline{\tau}_2 \in \mathfrak{T}_n$ be the rooted subtree consisting of those sequences of length at most $2$, so $\operatorname{Lea}(\overline{\tau}_2)$ is the set of those sequences in $n^*$ of length $2$.  The dihedral group of order $2n^2$ is generated by two involutions and includes a cycle of length $n^2$.  A finite symmetric group can be generated by a maximal length cycle together with a transposition.  Therefore, the group $\operatorname{Sym}(\operatorname{Lea}(\overline{\tau}_2))$ is generated by three involutions.  Take $H$ to be the subgroup of $V_n$ which is generated by these three involutions together with the transposition $(1\hspace{.1cm} 00)$.  It will suffice to prove that for incomparable $\alpha$ and $\beta$ the transposition $(\alpha \hspace{.1cm}\beta)$ is in $H$.

If incomparable $\alpha$ and $\beta$ have length $2$, then of course $$(\alpha\hspace{.1cm} \beta) \in \operatorname{Sym}(\operatorname{Lea}(\overline{\tau}_2)) \leq H.$$  Similarly if both lengths $|\alpha|$, $|\beta|$ equal $1$ we have $(\alpha\hspace{.1cm} \beta) = \prod_{i < n}(\alpha i\hspace{.1cm} \beta i) \in \operatorname{Sym}(\operatorname{Lea}(\overline{\tau}_2)) \leq H$.  Thus $\operatorname{Sym}(\operatorname{Lea}(\overline{\tau}_1)) \leq H$.  If $|\alpha| = 1$ and $|\beta| = 2$ then let $\sigma_1 = (\alpha \hspace{.1cm} 1)$ and take $\sigma_2 \in \operatorname{Sym}(\operatorname{Lea}(\overline{\tau}_2))$ to move $(\beta)\sigma_1$ to $00$ and to fix $1i$ for each $i < n$.  Then $(\alpha\hspace{.1cm}\beta) = \sigma_1\sigma_2(1 \hspace{.1cm} 00)\sigma_2^{-1}\sigma_1^{-1} \in H$.

Now, suppose for induction that for incomparable $\alpha, \beta \in n^*$ with $|\alpha|, |\beta| \leq k$ the transposition $(\alpha\hspace{.1cm}\beta)$ is in $H$, where $k \geq 2$.  Let $\alpha, \beta \in n^*$ be incomparable with $|\alpha| \leq |\beta| = k + 1$.  Take $\beta' \preceq \beta$ with $|\beta'| = 2$, and as $n \geq 3$, let $\gamma \in n^*$ have length $1$ and be incomparable to both $\alpha$ and $\beta$.  If $\sigma_1 = (\beta'\hspace{.1cm}\gamma)$ has $|(\alpha)\sigma_1| \leq k$ then $(\alpha\hspace{.1cm}\beta) = \sigma_1((\alpha)\sigma_1 \hspace{.1cm} (\beta)\sigma_1)\sigma_1^{-1} \in H$.  Otherwise we know $|\alpha| = k + 1$ and $\alpha$ is incomparable to $\beta'$. In this case we let $\beta'' \preceq \beta'$ with $|\beta''| = 1$ and $\alpha' \preceq \alpha$ with $|\alpha'| = 2$.

If $\beta'' \preceq \alpha$ then take $\delta \in n^*$ with $|\delta| = 1$ and such that the elements $\beta'', \gamma, \delta$ are pairwise incomparable, let $\sigma_2 = (\alpha' \hspace{.1cm} \delta)$, and note that $$(\alpha\hspace{.1cm} \beta) = \sigma_1\sigma_2((\alpha)\sigma_2 \hspace{.1cm}(\beta)\sigma_1)\sigma_2^{-1}\sigma_1^{-1}$$ is an element of $H$.  Otherwise we let $\sigma_2 = (\alpha'\hspace{.1cm}\beta'')$ and once again the above equality holds and $(\alpha\ \beta)$ is an element of $H$.  Thus by induction, each transposition of incomparable elements of $n^*$ is in $H$, and so $H = V_n$.
\end{proof}

\section{Grothendieck pairs and flexibility}\label{sec:Flexibility}

Long and Reid introduced in \cite{Long11} the following definition.
\begin{defn}
A finitely generated residually finite group $G$ is \emph{Grothendieck rigid} if for every finitely generated proper subgroup $H\subset G$ the inclusion induced map $\widehat{H}\to\widehat{G}$ is not an isomorphism.
\end{defn}

We require the following definition.
\begin{defn}
For $i=1,\ldots, d$, let $\varphi_i\colon G_i\rightarrow Q$ be a homomorphism. The \emph{fibre product} of this family of maps is defined as
$$P_d:=\left\{(g_1,\ldots, g_d)\in G_1\times\ldots\times G_d\mid \varphi_i(g_i)=\varphi_j(g_j), i,j=1,\ldots, d\right\}.$$

\end{defn}

The following lemma gives us a criterion for the finite generation of fibre products.
\begin{lemma}\emph{\cite[Lemma 4.3]{Bridson2023}}\label{BridsonFibreProduct}
For $i=1,\ldots, d$, let $\pi_i\colon G_i\twoheadrightarrow Q$ be an epimorphism. If the groups $G_i$ are finitely generated and $Q$ is finitely presented, then the fibre product $P_d\subseteq G_1\times\ldots\times G_d$ is finitely generated.
\end{lemma}

Before we turn to the proof of profinite flexibility of some graph products we need the following result.
\begin{thm}\emph{\cite[Theorem 4.6]{Bridson2023}}\label{BridsonFlexibility}
For $i=1,\ldots,d$, let $G_i\twoheadrightarrow Q$ be an epimorphism of finitely generated groups, and let $P_d\subseteq G_1\times\ldots\times G_d$ be the associated fibre product. If $Q$ is finitely presented, $\widehat{Q}=1$ and $H_2(Q; \Z)=0$, then the inclusion $\iota\colon P_d\hookrightarrow G_1\times\ldots\times G_d$ induces an isomorphism $\widehat{\iota}\colon\widehat{P_d}\to\widehat{G_1}\times\ldots\times \widehat{G_d}$.
\end{thm}

\begin{thm}
\label{UniversalCoxeter}
Let $d\geq 2$. For $n\in\mathbb{N}$, $n\geq 2$, there exists $\ell_n\in\mathbb{N}$ such that $$\prod_{i=1}^d\left(\bigast_{i=1}^{\ell_n}C_n\right)$$
is not Grothendieck rigid and not profinitely rigid.
\end{thm}
\begin{proof}
Let $V$ denote Richard Thompson's simple group $V$.  By \cite{Higman1974}, $V$ is finitely presented and by \cite[Theorem 0.1]{Kapoudjian2002} $H_2(V;\ZZ)=0$ (in fact $V$ is acyclic \cite{SzymikWahl2019}).  Let $X_n=\{g\in V\mid \mathrm{ord}(g)=n \}$.   Now, $\langle X_n\rangle\trianglelefteq V$ but every finite group embeds into $V$, see \cite[page 241]{CannonFloydParry1996}. So $\langle X_n\rangle$ is a non-trivial normal subgroup of a simple group.  Hence, $\langle X_n\rangle = V$. Thus, there is a finite set $Y_n\subseteq X_n$ of cardinality $\ell_n$ such that $Y_n$ generates $V$.

Let $H_n=\bigast_{i=1}^{\ell_n} C_n$ and $H_n^d=\prod_{i=1}^d H_n$.  We have a surjection $\pi_n\colon H_n\onto V$ sending a generator of each factor $C_n$ to a distinct element of $Y_n$.  Since $V$ is infinite $N_n\coloneqq \ker \pi_n$ has infinite index.  Now, a classical result of Baumslag \cite[\textsection 6]{Baumslag1966} implies that $\ker\pi_n$ is not finitely generated.

For $d\geq 2$ we denote by $P_{n,d}$ the $d$-fold fibre product of the map $\pi_n$.
By \Cref{BridsonFibreProduct} we see that $P_{n,d}$ is finitely generated.  Now, \Cref{BridsonFlexibility} implies that the inclusion $P_{n,d}\hookrightarrow H_n^d$ induces an isomorphism $\widehat{P}_{n,d}\to \widehat{H}_n^d$.  

To conclude we need to show $P_{n,d}$ is not isomorphic to $H_n^d$.  Suppose for a contradiction that they are.  By the Kurosh subgroup theorem \cite{Kurosh1934} centralisers in $H_n$ are cyclic.  Now, it follows that the non-abelian subgroups of $H_n^d$ that are centralisers of non-cyclic groups are (up to relabelling factors) of the form $H_n^{d-1}\times \{1\}$.  The subgroups (up to relabelling factors) of the form $N_n^{d-1}\times \{1\}$ of $P_{n,d}$ are characterised similarly.
But these latter groups are not finitely generated. A contradiction.  We conclude $P_{n,d}$ is not isomorphic to $H_n^d$ and so $H_n^d$ is neither Grothendieck rigid nor profinitely rigid in the absolute sense.
\end{proof}

\begin{remark}\label{not.fp}
    The groups $P_{n,d}$ (for $n\geq2$ and $d\geq2$) are not finitely presented.  There are a number of ways to see this, here is one: $H=\prod_{i=1}^d\bigast_{j=1}^{\ell_n}C_n$ has a finite index subgroup $G$ isomorphic to a direct product of $d$ copies of a finitely generated free group $F$.  Restricting the map $\bigast_{j=1}^{\ell_n} C_n\to V$ to $F$ we obtain a fibre product $P_{n,d}'\leqslant G$ which is a finite index subgroup of $P_{n,d}$ and by the argument above a Grothendieck pair for $G$.  But every finitely presented subgroup of a direct product of free groups is separable \cite{BridsonWilton2008} and so cannot be a Grothendieck pair.  Hence, $P_{n,d}'$ is not finitely presented and so neither is $P_{n,d}$.
\end{remark}

\begin{lemma}\label{lemma_fibre_prods_non_iso}
    Let $d\geq 2$.  Let $G_1,\dots,G_\ell$ be non-trivial finitely generated residually finite groups in which the centraliser of every nonabelian subgroup is trivial. Let $H=\bigast_{j=1}^\ell G_j$ and suppose $H\not\cong D_\infty$. Suppose there exist surjections $\phi_n\colon \bigast_{j=1}^\ell G_j\to V_n$ for every $n\geq 2$.  Let $P_{n}$ denote the $d$-fold fibre product of $\phi_n$.  Then, $P_{n}\cong P_{m}$ if and only if $n=m$.
\end{lemma}
\begin{proof}
Assume that $P_n$ and $P_m$ are isomorphic and let $N_n = \ker(\phi_n)$ and $N_m = \ker(\phi_m)$.  We know that $\prod_{i = 1}^d N_n \leq P_n \leq H^d$.  For $1 \leq i \leq d$ let $p_i: H^d \rightarrow H$ denote projection to the $i$ coordinate.  If $K \leq P_n$ is nonabelian then for some $i$ we have $p_i(K) \leq H$ is nonabelian, so assume without loss of generality that $i = 1$.  The centraliser in $H$ of $p_1(K)$, denoted $C_H(p_1(K))$, is trivial and this implies that $C_{H^d}(K)$ is a subgroup of $\{1\} \times \prod_{i = 2}^d H$.  Therefore $C_{P_n}(K) = P_n \cap C_{H^d}(K)$ is a subgroup of $\{1\} \times \prod_{i = 2}^d N_n$.  What we have just shown is that if $K \leq P_n$ is nonabelian then there is some $1 \leq i_0 \leq d$ for which $C_{P_n}(K) \leq (\prod_{i \neq i_0} N_n) \times \{1\}$.


Now, $N_n$ is a subgroup of the free product $H$, so by the Kurosh subgroup theorem \cite{Kurosh1934} if $N_n$ were abelian then it would be either infinite cyclic or conjugate to a subgroup of one of the factors of $H$. The subgroup $N_n$ can not be contained in a free factor, because it is normal in $H$ and can not be cyclic, because $H$ is neither cyclic and nor $D_\infty$.

We therefore have 
\[C_{P_n}\left(N_n \times \prod_{i \neq i_0}\{1\}\right) = \left(\prod_{i \neq i_0} N_n \right) \times \{1\}.\]
Thus, each subgroup $(\prod_{i_0 \neq 1} N_n) \times \{1\}$ is an element in the set $\mathcal{X} = \{C_{P_n}(K) \mid K \leq P_n \text{ nonabelian} \}$ and each element in $\mathcal{X}$ includes into such a subgroup.  Therefore the subgroup of $P_n$ generated by maximal subgroups which centralize a nonabelian subgroup is precisely $\prod_{i = 1}^d N_n$, and by the same reasoning the comparably defined subgroup in $P_m$ is $\prod_{i = 1}^d N_m$.  As $P_n$ is isomorphic to $P_m$, $\prod_{i = 1}^d N_n$ corresponds to $\prod_{i = 1}^d N_m$ under the isomorphism.  Therefore $P_n/\prod_{i = 1}^d N_n \cong V_n$ is isomorphic to $P_m/\prod_{i = 1}^d N_m \cong V_m$, and $V_n \simeq V_m$ implies that $n = m$ by \cite[Theorem 6.4]{Higman1974}.
\end{proof}

\begin{thm}
\label{infinitegenus}  
Let $G_1, \ldots, G_{\ell}$ be finitely generated, residually finite groups such that for all $1 \leq j \leq \ell$ the centraliser of a nonabelian subgroup of $G_j$ is trivial.  Assume also that at least four of the $G_j$ have a subgroup of index $2$.  Let $d \geq 2$.  The genus of $$\prod_{i=1}^d\left(\bigast_{j=1}^{\ell}G_j\right)$$ is infinite.
\end{thm}

\begin{proof}  Fix $d \geq 2$.  Let $H = \bigast_{j=1}^{\ell}G_j$.  We have that $H$ is residually finite and finitely generated, that the centraliser of a nonabelian subgroup of $H$ is trivial (by the Kurosh subgroup theorem \cite{Kurosh1934}), and that for each even $n \geq 2$ we have a surjective homomorphism $\phi_n: H \rightarrow V_n$ (by \Cref{fourinvolutionsgenerate}).  For each even $n \geq 2$ let $P_n$ 
denote the $d$-fold fibre product of the map $\phi_n$ (since we fix $d$, we omit it from the notation here).  Further, $V_n$ is finitely presented by \cite{Higman1974} and $H_2(V_n;\Z)=0$ by \cite[Theorem 0.1]{Kapoudjian2002}. Hence, by \Cref{BridsonFlexibility} we see that $\widehat{P_n}$ is isomorphic to $\widehat{\prod_{i = 1}^dH}$ and by \Cref{BridsonFibreProduct} each $P_n$ is finitely generated.  Now, \Cref{lemma_fibre_prods_non_iso} implies that $P_n \cong P_m$ if and only if $n = m$, whence the theorem.
\end{proof}

Recall that a group is \emph{indicable} if it has a surjective homomorphism onto the integers.  A group $J$ is \emph{3/2-generated} if for every nontrivial $g \in J$ there exists a $g' \in J$ such that $\{g, g'\}$ generates $J$.

\begin{thm}\label{nicefreeprod} Let $G$ and $L$ be finitely generated, residually finite groups in which the centraliser of a nonabelian subgroup is trivial.  Assume further that $G$ is nontrivial and $L$ is indicable.  Then the product $\prod_{i=1}^d(G \bigast L)$, where $d \geq 2$, has infinite genus.
\end{thm}

\begin{proof} Fix $d \geq 2$.  Let $H$ denote the free product $G \bigast L$.  We know that $V_n$ is $3/2$-generated for each natural number $n \geq 2$ \cite[Theorem 1]{DonovenHarper2020} and includes every finite group.  Since $G$ is nontrivial and residually finite, and $L$ is indicable we have for each even $n \geq 2$ a surjective homomorphism $\phi_n: H \rightarrow V_n$.  Take $P_n \leq \prod_{i = 1}^d H$ to be the $d$-fold fibre product of $\phi_n$ and $N_n$ to be its kernel.  As before, the inclusion map $\iota: P_n \hookrightarrow \prod_{i = 1}^d H$ induces an isomorphism $\widehat{\iota}\colon\widehat{P_n}\to \prod_{i = 1}^{d}\widehat{H}$. Now, \Cref{lemma_fibre_prods_non_iso} implies that $P_n \cong P_m$ if and only if $n = m$, whence the theorem.
\end{proof}

For a straightforward application of Theorem \ref{nicefreeprod}, one can take $L = \mathbb{Z}$ and take $G$ to be a nontrivial finitely generated abelian group (say, a nontrivial cyclic group).

For another application one can take $L = \mathbb{Z}$ and $G$ to be a nontrivial torsion-free group with finite $C'(1/6)$ presentation.  Such a $G$ is famously residually finite since it is hyperbolic and by combining \cite[Theorem 1.2]{Wise2004} with \cite[Corollary 1.2]{Agol2013}. Furthermore, the centraliser in a torsion-free hyperbolic group of a nontrivial element is cyclic, so $G$ satisfies the hypotheses of Theorem \ref{nicefreeprod}.  Random groups in the few-relator sense are torsion-free and $C'(1/6)$.  More concretely $G$ can be the fundamental group of an orientable surface of genus at least $2$.

Our last application is a generalisation of \cite[Proposition 9.2]{BridsonGrunewald2004}.
\begin{corollary}
\label{freegroupsinfinitegenus}
    Let $d\in\mathbb{N}$, $d\geq 2$. The direct product $\prod_{i=1}^d F_n$ of free groups of rank $n$ has infinite genus if and only if $n\geq 2$.
\end{corollary}
\begin{proof}
If $n=1$, then it is known that $\mathbb{Z}^d$ has genus $1$. Assume that $n\geq 2$. We take $G=F_{n-1}$ and $L=\Z$. The group $G$ is a torsion free hyperbolic group, hence the centraliser of a nonabelian subgroup is trivial. Thus by \Cref{nicefreeprod} the genus of $\prod_{i=1}^d F_n$ is infinite. 
\end{proof}

The following definitions follow \cite{Radcliffe2003}. Let $\Gamma$ be a graph and let $\Omega$ be a subgraph.  We say $\Omega$ is a \emph{module} if $\Omega$ is a full subgraph and if for each edge $\left\{v,k\right\}\in E(\Gamma)$ with $v \in V(\Gamma)\backslash V(\Omega)$ and $k\in \Omega$, and for each $k'\in \Omega$, there exists an edge $\left\{v,k'\right\}\in E(\Gamma)$. 
We form the \emph{collapsed graph} $\Gamma'$ by taking the vertex set to be $V(\Gamma)\backslash V(\Omega) \cup \{\ast\}$ and the edge set to be
    \[
    \left\{\{v,w\}\in E(\Gamma)\mid v,w\in V(\Gamma)\backslash V(\Omega)\right\} \cup \left\{(v,\ast)\mid v\in V(\Gamma),\ \{v,k\}\in E(\Gamma),\ k\in \Omega \right\}.
    \]
See \Cref{fig:module} for an example.

\begin{figure}[h]
\begin{center}
\begin{tikzpicture}

    \coordinate (A) at (0,0);
    \coordinate (B) at (0,2);
    \coordinate (C) at (2,0);
    \coordinate (D) at (2,2);
    \coordinate (E) at (-2,1);
    \coordinate (F) at (-3,1);
    \coordinate (G) at (4,1);

    \coordinate (H) at (6,1);
    \coordinate (I) at (7,1);
    \coordinate[label=below: {$\ast$}] (J) at (8,1);
    \coordinate (K) at (9,1);

    \draw[color=blue] (A) -- (B);
    \draw[color=blue] (A) -- (C);
    \draw[color=blue] (B) -- (D);
    \draw[color=blue] (C) -- (D);

    \draw (1,2.5) -- (1,2.5) node[above]{$L$};
    
    \draw[color=blue] (1,0) -- (1,0) node[below]{$K$};

    \draw (7.5,2.5) -- (7.5,2.5) node[above]{$L'$}; 
    
    \draw (F) -- (E);
    \draw (E) -- (A);
    \draw (E) -- (B);
    \draw (E) -- (C);
    \draw (E) -- (D);

    \draw (A) -- (G);
    \draw (B) -- (G);
    \draw (C) -- (G);
    \draw (D) -- (G);

    \draw (H) -- (I);
    \draw (I) -- (J);
    \draw (J) -- (K);

    \fill[color=blue] (A) circle (2pt);
    \fill[color=blue] (B) circle (2pt);
    \fill[color=blue] (C) circle (2pt);
    \fill[color=blue] (D) circle (2pt);
    \fill (E) circle (2pt);
    \fill (F) circle (2pt);
    \fill (G) circle (2pt);
    \fill (H) circle (2pt);
    \fill (I) circle (2pt);
    \fill (J) circle (2pt);
    \fill (K) circle (2pt);
\end{tikzpicture}
\caption{}
\label{fig:module}
\end{center}
\end{figure}

There is an isomorphism of the graph product $G_\Gamma$ with the graph product $G_{\Gamma'}$ when the group $G_\ast$ is isomorphic to $G_\Omega$.
Suppose there exists a module $\Omega$ of $\Gamma$ and let $G_{\Gamma'}$ denote the collapsed graph.  We say $G_\Omega$ is \emph{$G_\Gamma$-stably isomorphic} to another group $P$ (not isomorphic to $G_\Omega$) if $G_\Gamma$ is isomorphic to the graph product $G_{\Gamma'}$ where the group $G_\ast$ is isomorphic to $P$.

We say a group $G$ is \emph{profinitely flexible with $P$} if $P$ is a finitely generated residually finite group such that $G\not\cong P$ but $\widehat{G}\cong\widehat{P}$.

\begin{prop}\label{prof.flex.graph.prod}
    Let $\Gamma$ be a graph and let $\Omega$ be a module.  Let $G_\Gamma$ be a graph product on $\Gamma$ and suppose $G_\Omega$ is profinitely flexible with $P$.  If $P$ is not stably $G_\Gamma$-isomorphic to $H$, then $G$ is profinitely flexible.
\end{prop}
\begin{proof}
    A model for $\widehat{G}_\Gamma$ is given by taking the graph product of $G_\Gamma$ in the category of profinite groups.  Namely, each vertex group $G_v$ is replaced by $\widehat G_v$, we form the profinite free product $\bigast_{v\in V(L)} \widehat G_v$, and then quotient out by the set of relations $[g_v, g_w]$ for $g_v\in\widehat G_v$ and $g_w\in \widehat G_w$, whenever $\{v,w\}$ is an edge of $\Gamma$.

    Now, $\Omega$ is collapsible in $\Gamma$. It follows there is an isomorphism of $G_\Gamma$ with the graph product $G_{\Gamma'}$ when the group $G_\ast$ is isomorphic to $G_\Omega$.
    On the level of profinite completions this means that $\widehat G_\Gamma$ is isomorphic to $\widehat G_{\Gamma'}$.  Replacing the group $G_\ast$ with $P$ we obtain a graph product $G_{\Gamma'}$ not isomorphic to $G_\Gamma$.  But, on the level of profinite completions $\widehat P$ and $\widehat G_\Omega$ are isomorphic.  Hence, $\widehat G_\Gamma$ and $\widehat G_{\Gamma'}$ are isomorphic. So $G_\Gamma$ is profinitely flexible.
\end{proof}

\begin{remark}
    The stability hypothesis is indeed necessary.  Upon dropping the hypothesis one can easily construct counter examples by considering graph products with vertex groups isomorphic to $\Z$ or the metacyclic groups considered by Baumslag \cite{Baumslag1974}. 
\end{remark}

\begin{corollary}
    Let $\Gamma$ be a graph and let $\Omega$ be a $(4,4)$-bipartite module.  Then, $W_\Gamma$ is profinitely flexible.
\end{corollary}
\begin{proof}
    In \Cref{UniversalCoxeter} we showed that $W_\Omega$ is profinitely flexible with pair $P$.  Thus, by \Cref{prof.flex.graph.prod} it suffices to show that $P$ is not stably $W_\Gamma$-isomorphic to $W_\Omega$.  But this is immediate since $P$ was not finitely presented (see \Cref{not.fp}).
\end{proof}

\begin{corollary}
    Let $\Gamma$ be a graph and let $\Omega$ be a $(2,2)$-bipartite module.  Then $A_\Gamma$ is profinitely flexible.
\end{corollary}
\begin{proof}
    The proof is analogous.  The key difference is profinite flexibility of $F_2\times F_2$ was shown in \cite{PlatonovTavgen1986}.
\end{proof}

\begin{corollary}\label{rem.irr}
There exist right-angled Coxeter and Artin groups that are profinitely flexible but which are neither freely nor directly reducible.
\end{corollary}
\begin{proof}
    We apply the previous two corollaries to the graph in \Cref{fig:module}.  In the RACG case we take the graph product with vertex groups equal to $D_\infty=C_2\bigast C_2$ and for the RAAG case we take the graph product with vertex groups equal to $\ZZ$.
\end{proof}

\bibliographystyle{S2957_halpha_v2}
\bibliography{S2957_refs_v2.bib}
\newpage

\end{document}

%% file: S2957_top_v2.tex
\newcounter{commentcounter}

\usepackage{amsmath} 
\usepackage{amssymb} 
\usepackage{amsthm} 
\usepackage{stmaryrd} 
\usepackage[english]{babel} 
\usepackage[font=small,justification=centering]{caption} 
\usepackage[nodayofweek]{datetime}
\usepackage[shortlabels]{enumitem} 
\usepackage[T1]{fontenc} 
\usepackage[utf8]{inputenc} 
\usepackage{ifthen} 
\usepackage{mathabx} 
\usepackage{mathtools} 
\usepackage[dvipsnames]{xcolor} 
\usepackage{qtree}
\usepackage[pdftex,  colorlinks=true, pagebackref=true]{hyperref} 
    \hypersetup{urlcolor=RoyalBlue, linkcolor=RoyalBlue,  citecolor=black}
\usepackage{setspace} 
\usepackage{tikz-cd} 
\usetikzlibrary{shapes.geometric} 
\usepackage{xfrac} 
\usepackage[capitalize]{cleveref} 
\renewcommand*{\backref}[1]{}
\renewcommand*{\backrefalt}[4]
{
    \ifcase #1
        No citation in the text.
    \or
        Cited on Page #2.
    \else
        Cited on Pages #2.
    \fi
}

\makeatletter
\@namedef{subjclassname@1991}{Mathematical subject classification 1991}
\@namedef{subjclassname@2000}{Mathematical subject classification 2000}
\@namedef{subjclassname@2010}{Mathematical subject classification 2010}
\@namedef{subjclassname@2020}{Mathematical subject classification 2020}
\makeatother

\newtheorem{thm}{Theorem}[section]
\newtheorem{lemma}[thm]{Lemma}
\newtheorem{corollary}[thm]{Corollary}
\newtheorem{prop}[thm]{Proposition}

\newtheorem*{claim*}{Claim}

\newtheorem{question}[thm]{Question}

\theoremstyle{definition}
\newtheorem{defn}[thm]{Definition}
\newtheorem{remark}[thm]{Remark}

\theoremstyle{plain}

    \newtheoremstyle{TheoremNum}
        {\topsep}{\topsep} 
        {\itshape} 
        {-0.15cm} 
        {\bfseries} 
        {.} 
        { }  
        {\thmname{#1}\thmnote{ \bfseries #3}}
    \theoremstyle{TheoremNum}
    \newtheorem{duplicate}{}

\newcommand*{\claimproofname}{My proof}
\newenvironment{claimproof}[1][\claimproofname]{\begin{proof}[#1]}{\end{proof}}

\DeclareMathOperator{\st}{\mathrm{st}}
\DeclareMathOperator{\lk}{\mathrm{lk}}
\DeclareMathOperator{\Pc}{\mathrm{Pc}}

\newcommand{\caln}{{\mathcal{N}}}

\newcommand{\onto}{\twoheadrightarrow}

\def\Z{\mathbb{Z}}

\newcommand{\ZZ}{\mathbb{Z}}

\usepackage{tikz}
\usetikzlibrary{arrows,quotes}
\tikzstyle{blackNode}=[fill=black, draw=black, shape=circle]

%% file: S2957_v2.bbl
\begin{thebibliography}{BMRS21}

\bibitem[Ago13]{Agol2013}
Ian Agol.
\newblock The virtual {H}aken conjecture.
\newblock {\em Doc. Math.}, 18:1045--1087, 2013.
\newblock With an appendix by Agol, Daniel Groves, and Jason Manning,
  \href{https://dx.doi.org/10.4171/DM/421}{{\ttfamily 10.4171/DM/421}}.

\bibitem[AM15]{AntolinMinasyan2015}
Yago Antol\'{\i}n and Ashot Minasyan.
\newblock Tits alternatives for graph products.
\newblock {\em J. Reine Angew. Math.}, 704:55--83, 2015.
\newblock \href{https://dx.doi.org/10.1515/crelle-2013-0062}{{\ttfamily
  10.1515/crelle-2013-0062}}.

\bibitem[BL00]{BassLubotzky2000}
Hyman Bass and Alexander Lubotzky.
\newblock Nonarithmetic superrigid groups: counterexamples to {P}latonov's
  conjecture.
\newblock {\em Ann. of Math. (2)}, 151(3):1151--1173, 2000.
\newblock \href{https://dx.doi.org/10.2307/121131}{{\ttfamily 10.2307/121131}}.

\bibitem[Bau66]{Baumslag1966}
Benjamin Baumslag.
\newblock Intersections of finitely generated subgroups in free products.
\newblock {\em J. London Math. Soc.}, 41:673--679, 1966.
\newblock \href{https://dx.doi.org/10.1112/jlms/s1-41.1.673}{{\ttfamily
  10.1112/jlms/s1-41.1.673}}.

\bibitem[Bau74]{Baumslag1974}
Gilbert Baumslag.
\newblock Residually finite groups with the same finite images.
\newblock {\em Compositio Math.}, 29:249--252, 1974.

\bibitem[BGZ14]{BessaGrunewaldZalesskii2014}
Vagner Bessa, Fritz Grunewald, and Pavel~A. Zalesskii.
\newblock Genus for virtually surface groups and pullbacks.
\newblock {\em Manuscripta Math.}, 145(1-2):221--233, 2014.
\newblock \href{https://dx.doi.org/10.1007/s00229-014-0677-7}{{\ttfamily
  10.1007/s00229-014-0677-7}}.

\bibitem[BQ17]{BleakQuick2017}
Collin Bleak and Martyn Quick.
\newblock The infinite simple group {$V$} of {R}ichard {J}. {T}hompson:
  presentations by permutations.
\newblock {\em Groups Geom. Dyn.}, 11(4):1401--1436, 2017.
\newblock \href{https://dx.doi.org/10.4171/GGD/433}{{\ttfamily
  10.4171/GGD/433}}.

\bibitem[Bou68]{Bourbaki68}
N.~Bourbaki.
\newblock {\em \'{E}l\'{e}ments de math\'{e}matique. {F}asc. {XXXIV}. {G}roupes
  et alg\`ebres de {L}ie. {C}hapitre {IV}: {G}roupes de {C}oxeter et syst\`emes
  de {T}its. {C}hapitre {V}: {G}roupes engendr\'{e}s par des r\'{e}flexions.
  {C}hapitre {VI}: syst\`emes de racines}, volume No. 1337 of {\em
  Actualit\'{e}s Scientifiques et Industrielles [Current Scientific and
  Industrial Topics]}.
\newblock Hermann, Paris, 1968.

\bibitem[BMRS20]{BridsonMcReynoldsReidSpitler2020}
Martin Bridson, David~Ben McReynolds, Alan Reid, and Ryan Spitler.
\newblock Absolute profinite rigidity and hyperbolic geometry.
\newblock {\em Ann. Math. (2)}, 192(3):679--719, 2020.
\newblock \href{https://dx.doi.org/10.4007/annals.2020.192.3.1}{{\ttfamily
  10.4007/annals.2020.192.3.1}}.

\bibitem[Bri16]{Bridson2016}
Martin~R. Bridson.
\newblock The strong profinite genus of a finitely presented group can be
  infinite.
\newblock {\em J. Eur. Math. Soc. (JEMS)}, 18(9):1909--1918, 2016.
\newblock \href{https://dx.doi.org/10.4171/JEMS/633}{{\ttfamily
  10.4171/JEMS/633}}.

\bibitem[Bri24]{Bridson2023}
Martin~R. Bridson.
\newblock Profinite isomorphisms and fixed-point properties.
\newblock {\em Algebr. Geom. Topol.}, 24(7):4103--4114, 2024.
\newblock \href{https://dx.doi.org/10.2140/agt.2024.24.4103}{{\ttfamily
  10.2140/agt.2024.24.4103}}.

\bibitem[BG04]{BridsonGrunewald2004}
Martin~R. Bridson and Fritz~J. Grunewald.
\newblock Grothendieck's problems concerning profinite completions and
  representations of groups.
\newblock {\em Ann. of Math. (2)}, 160(1):359--373, 2004.
\newblock \href{https://dx.doi.org/10.4007/annals.2004.160.359}{{\ttfamily
  10.4007/annals.2004.160.359}}.

\bibitem[BMRS21]{BridsonMcReynoldsReidSpitler2021}
Martin~R. Bridson, D.~B. McReynolds, Alan~W. Reid, and Ryan Spitler.
\newblock On the profinite rigidity of triangle groups.
\newblock {\em Bull. Lond. Math. Soc.}, 53(6):1849--1862, 2021.
\newblock \href{https://dx.doi.org/10.1112/blms.12546}{{\ttfamily
  10.1112/blms.12546}}.

\bibitem[BW08]{BridsonWilton2008}
Martin~R. Bridson and Henry Wilton.
\newblock Subgroup separability in residually free groups.
\newblock {\em Math. Z.}, 260(1):25--30, 2008.
\newblock \href{https://dx.doi.org/10.1007/s00209-007-0256-7}{{\ttfamily
  10.1007/s00209-007-0256-7}}.

\bibitem[CFP96]{CannonFloydParry1996}
J.~W. Cannon, W.~J. Floyd, and W.~R. Parry.
\newblock Introductory notes on {R}ichard {T}hompson's groups.
\newblock {\em Enseign. Math.}, 42:215--256, 1996.

\bibitem[CM13]{CapraceMinasyan2013}
Pierre-Emmanuel Caprace and Ashot Minasyan.
\newblock On conjugacy separability of some {C}oxeter groups and
  parabolic-preserving automorphisms.
\newblock {\em Illinois J. Math.}, 57(2):499--523, 2013.
\newblock \href{https://dx.doi.org/10.1215/ijm/1408453592}{{\ttfamily
  10.1215/ijm/1408453592}}.

\bibitem[CZ16]{ChagasZalesskii2016}
Sheila~C. Chagas and Pavel~A. Zalesskii.
\newblock Limit groups are subgroup conjugacy separable.
\newblock {\em J. Algebra}, 461:121--128, 2016.
\newblock \href{https://dx.doi.org/10.1016/j.jalgebra.2016.04.026}{{\ttfamily
  10.1016/j.jalgebra.2016.04.026}}.

\bibitem[CW24]{CheethamWest2022}
Tamunonye Cheetham-West.
\newblock Absolute profinite rigidity of some closed fibered hyperbolic
  3-manifolds.
\newblock {\em Math. Res. Lett.}, 31(3):615--638, 2024.
\newblock \href{https://dx.doi.org/10.4310/mrl.241113034702}{{\ttfamily
  10.4310/mrl.241113034702}}.

\bibitem[Dav08]{Davis2008}
Michael~W. Davis.
\newblock {\em The geometry and topology of {C}oxeter groups}, volume~32 of
  {\em London Mathematical Society Monographs Series}.
\newblock Princeton University Press, Princeton, NJ, 2008.

\bibitem[dBPZ22]{deBessaPortoZalesskii2022}
V.~R. de~Bessa, A.~L.~P. Porto, and P.~A Zalesskii.
\newblock The profinite completion of accessible groups.
\newblock {\em Monatshefte für Mathematik}, 2022.
\newblock \href{https://dx.doi.org/10.1007/s00605-022-01789-9}{{\ttfamily
  10.1007/s00605-022-01789-9}}.

\bibitem[DFPR82]{Dixon1982}
J.~D. Dixon, E.~W. Formanek, J.~C. Poland, and L.~Ribes.
\newblock Profinite completions and isomorphic finite quotients.
\newblock {\em Journal of Pure and Applied Algebra}, 23:227--231, 1982.

\bibitem[DH20]{DonovenHarper2020}
Casey Donoven and Scott Harper.
\newblock Infinite {$\frac 32$}-generated groups.
\newblock {\em Bull. Lond. Math. Soc.}, 52(4):657--673, 2020.
\newblock \href{https://dx.doi.org/10.1112/blms.12356}{{\ttfamily
  10.1112/blms.12356}}.

\bibitem[Fer16]{Ferov2016}
Michal Ferov.
\newblock On conjugacy separability of graph products of groups.
\newblock {\em J. Algebra}, 447:135--182, 2016.
\newblock \href{https://dx.doi.org/10.1016/j.jalgebra.2015.08.027}{{\ttfamily
  10.1016/j.jalgebra.2015.08.027}}.

\bibitem[Fun13]{Funar2013}
Louis Funar.
\newblock Torus bundles not distinguished by {TQFT} invariants. {With} an
  appendix by {Louis} {Funar} and {Andrei} {Rapinchuk}.
\newblock {\em Geom. Topol.}, 17(4):2289--2344, 2013.
\newblock \href{https://dx.doi.org/10.2140/gt.2013.17.2289}{{\ttfamily
  10.2140/gt.2013.17.2289}}.

\bibitem[Gal05]{Gal2005}
\'{S}wiatos\l aw~R. Gal.
\newblock On normal subgroups of {C}oxeter groups generated by standard
  parabolic subgroups.
\newblock {\em Geom. Dedicata}, 115:65--78, 2005.
\newblock \href{https://dx.doi.org/10.1007/s10711-005-7890-1}{{\ttfamily
  10.1007/s10711-005-7890-1}}.

\bibitem[GM19]{GenevoisMartin2019}
Anthony Genevois and Alexandre Martin.
\newblock Automorphisms of graph products of groups from a geometric
  perspective.
\newblock {\em Proc. Lond. Math. Soc. (3)}, 119(6):1745--1779, 2019.
\newblock \href{https://dx.doi.org/10.1112/plms.12282}{{\ttfamily
  10.1112/plms.12282}}.

\bibitem[Gre90]{Green1990}
Elisabeth~Ruth Green.
\newblock {\em Graph products of groups}.
\newblock 1990.
\newblock Thesis (Ph.D.)--The University of Leeds.

\bibitem[GPS80]{GrunewaldPickelSegal1980}
F.~J. Grunewald, P.~F. Pickel, and D.~Segal.
\newblock Polycyclic groups with isomorphic finite quotients.
\newblock {\em Ann. of Math. (2)}, 111(1):155--195, 1980.
\newblock \href{https://dx.doi.org/10.2307/1971220}{{\ttfamily
  10.2307/1971220}}.

\bibitem[Hem14]{Hempel2014}
John Hempel.
\newblock Some 3-manifold groups with the same finite quotients, 2014.
\newblock \href{https://arxiv.org/abs/1409.3509}{{\ttfamily arXiv:1409.3509
  [math.GT]}}.

\bibitem[Hig74]{Higman1974}
Graham Higman.
\newblock {\em Finitely presented infinite simple groups}, volume No. 8 of {\em
  Notes on Pure Mathematics}.
\newblock Australian National University, Department of Pure Mathematics,
  Department of Mathematics, I.A.S., Canberra, 1974.

\bibitem[HK25a]{HughesKielak2022}
Sam Hughes and Dawid Kielak.
\newblock Profinite rigidity of fibring.
\newblock {\em Rev. Mat. Iberoam. (Online first)}, 2025.
\newblock \href{https://dx.doi.org/10.4171/RMI/1524}{{\ttfamily
  10.4171/RMI/1524}}.

\bibitem[HK25b]{HughesKudlinska2023}
Sam Hughes and Monika Kudlinska.
\newblock On profinite rigidity amongst free-by-cyclic groups {I}: the generic
  case.
\newblock {\em Proc. Lond. Math. Soc. (3)}, 130(6):Paper No. e70059, 43, 2025.
\newblock \href{https://dx.doi.org/10.1112/plms.70059}{{\ttfamily
  10.1112/plms.70059}}.

\bibitem[JZ20]{Jaikin2020}
Andrei Jaikin-Zapirain.
\newblock Recognition of being fibered for compact 3-manifolds.
\newblock {\em Geom. Topol.}, 24(1):409--420, 2020.
\newblock \href{https://dx.doi.org/10.2140/gt.2020.24.409}{{\ttfamily
  10.2140/gt.2020.24.409}}.

\bibitem[Kap02]{Kapoudjian2002}
Christophe Kapoudjian.
\newblock Virasoro-type extensions for the {H}igman-{T}hompson and {N}eretin
  groups.
\newblock {\em Q. J. Math.}, 53(3):295--317, 2002.
\newblock \href{https://dx.doi.org/10.1093/qjmath/53.3.295}{{\ttfamily
  10.1093/qjmath/53.3.295}}.

\bibitem[KW16]{KrophollerWilkes2016}
Robert Kropholler and Gareth Wilkes.
\newblock Profinite properties of {RAAG}s and special groups.
\newblock {\em Bull. Lond. Math. Soc.}, 48(6):1001--1007, 2016.
\newblock \href{https://dx.doi.org/10.1112/blms/bdw056}{{\ttfamily
  10.1112/blms/bdw056}}.

\bibitem[Kur34]{Kurosh1934}
Aleksandr~Gennadyevich Kurosh.
\newblock Die untergruppen der freien produkte von beliebigen gruppen.
\newblock {\em Mathematische Annalen}, 109:647--660, 1934.

\bibitem[Liu23]{Liu2023}
Yi~Liu.
\newblock Finite-volume hyperbolic 3-manifolds are almost determined by their
  finite quotient groups.
\newblock {\em Invent. Math.}, 231(2):741--804, 2023.
\newblock \href{https://dx.doi.org/10.1007/s00222-022-01155-4}{{\ttfamily
  10.1007/s00222-022-01155-4}}.

\bibitem[LR11]{Long11}
Darren~D. Long and Alan~W. Reid.
\newblock Grothendieck's problem for 3-manifold groups.
\newblock {\em Groups Geom. Dyn.}, 5(2):479--499, 2011.
\newblock \href{https://dx.doi.org/10.4171/GGD/135}{{\ttfamily
  10.4171/GGD/135}}.

\bibitem[Mal40]{Malcev1940}
A.~I. Malcev.
\newblock On the faithful representation of infinite groups by matrices [in
  russian].
\newblock {\em Mat. Sb.}, 8(3):405--423, 1940.

\bibitem[Min21]{Minasyan2021}
Ashot Minasyan.
\newblock Virtual retraction properties in groups.
\newblock {\em Int. Math. Res. Not. IMRN}, (17):13434--13477, 2021.
\newblock \href{https://dx.doi.org/10.1093/imrn/rnz249}{{\ttfamily
  10.1093/imrn/rnz249}}.

\bibitem[MV24]{MollerVarghese2023}
Philip M\"oller and Olga Varghese.
\newblock On quotients of {C}oxeter groups.
\newblock {\em J. Algebra}, 639:516--531, 2024.
\newblock \href{https://dx.doi.org/10.1016/j.jalgebra.2023.09.048}{{\ttfamily
  10.1016/j.jalgebra.2023.09.048}}.

\bibitem[Mü06]{Muhlherr2006}
Bernhard Mühlherr.
\newblock The isomorphism problem for {C}oxeter groups.
\newblock In {\em The {C}oxeter legacy}, pages 1--15. Amer. Math. Soc.,
  Providence, RI, 2006.

\bibitem[NS07a]{NikolovSegal2007I}
Nikolay Nikolov and Dan Segal.
\newblock On finitely generated profinite groups. {I}. {S}trong completeness
  and uniform bounds.
\newblock {\em Ann. of Math. (2)}, 165(1):171--238, 2007.
\newblock \href{https://dx.doi.org/10.4007/annals.2007.165.171}{{\ttfamily
  10.4007/annals.2007.165.171}}.

\bibitem[NS07b]{NikolovSegal2007II}
Nikolay Nikolov and Dan Segal.
\newblock On finitely generated profinite groups. {II}. {P}roducts in
  quasisimple groups.
\newblock {\em Ann. of Math. (2)}, 165(1):239--273, 2007.
\newblock \href{https://dx.doi.org/10.4007/annals.2007.165.239}{{\ttfamily
  10.4007/annals.2007.165.239}}.

\bibitem[Pic71]{Pickel1971}
P.~F. Pickel.
\newblock Finitely generated nilpotent groups with isomorphic finite quotients.
\newblock {\em Trans. Amer. Math. Soc.}, 160:327--341, 1971.
\newblock \href{https://dx.doi.org/10.2307/1995809}{{\ttfamily
  10.2307/1995809}}.

\bibitem[Pic74]{Pickel1974}
P.~F. Pickel.
\newblock Metabelian groups with the same finite quotients.
\newblock {\em Bull. Austral. Math. Soc.}, 11:115--120, 1974.
\newblock \href{https://dx.doi.org/10.1017/S0004972700043689}{{\ttfamily
  10.1017/S0004972700043689}}.

\bibitem[PT86]{PlatonovTavgen1986}
V.~P. Platonov and O.~I. Tavgen'.
\newblock On the {G}rothendieck problem of profinite completions of groups.
\newblock {\em Dokl. Akad. Nauk SSSR}, 288(5):1054--1058, 1986.

\bibitem[Pyb04]{Pyber2004}
L\'{a}szl\'{o} Pyber.
\newblock Groups of intermediate subgroup growth and a problem of
  {G}rothendieck.
\newblock {\em Duke Math. J.}, 121(1):169--188, 2004.
\newblock \href{https://dx.doi.org/10.1215/S0012-7094-04-12115-3}{{\ttfamily
  10.1215/S0012-7094-04-12115-3}}.

\bibitem[Qi07]{Qi2007}
Dongwen Qi.
\newblock A note on parabolic subgroups of a {C}oxeter group.
\newblock {\em Expo. Math.}, 25(1):77--81, 2007.
\newblock \href{https://dx.doi.org/10.1016/j.exmath.2006.05.001}{{\ttfamily
  10.1016/j.exmath.2006.05.001}}.

\bibitem[Rad03]{Radcliffe2003}
David~G. Radcliffe.
\newblock Rigidity of graph products of groups.
\newblock {\em Algebr. Geom. Topol.}, 3:1079--1088, 2003.
\newblock \href{https://dx.doi.org/10.2140/agt.2003.3.1079}{{\ttfamily
  10.2140/agt.2003.3.1079}}.

\bibitem[Rib17]{RibesBook}
Luis Ribes.
\newblock {\em Profinite graphs and groups}, volume~66 of {\em Ergebnisse der
  Mathematik und ihrer Grenzgebiete. 3. Folge. A Series of Modern Surveys in
  Mathematics [Results in Mathematics and Related Areas. 3rd Series. A Series
  of Modern Surveys in Mathematics]}.
\newblock Springer, Cham, 2017.
\newblock \href{https://dx.doi.org/10.1007/978-3-319-61199-0}{{\ttfamily
  10.1007/978-3-319-61199-0}}.

\bibitem[SRS24]{SantosRegoSchwer2022}
Yuri Santos~Rego and Petra Schwer.
\newblock The galaxy of {C}oxeter groups.
\newblock {\em J. Algebra}, 656:406--445, 2024.
\newblock \href{https://dx.doi.org/10.1016/j.jalgebra.2023.12.006}{{\ttfamily
  10.1016/j.jalgebra.2023.12.006}}.

\bibitem[Sol76]{Solomon76}
Louis Solomon.
\newblock A {M}ackey formula in the group ring of a {C}oxeter group.
\newblock {\em J. Algebra}, 41(2):255--264, 1976.
\newblock \href{https://dx.doi.org/10.1016/0021-8693(76)90182-4}{{\ttfamily
  10.1016/0021-8693(76)90182-4}}.

\bibitem[Ste72]{Stebe1972}
Peter~F. Stebe.
\newblock Conjugacy separability of groups of integer matrices.
\newblock {\em Proc. Am. Math. Soc.}, 32:1--7, 1972.
\newblock \href{https://dx.doi.org/10.2307/2038292}{{\ttfamily
  10.2307/2038292}}.

\bibitem[SW19]{SzymikWahl2019}
Markus Szymik and Nathalie Wahl.
\newblock The homology of the {H}igman-{T}hompson groups.
\newblock {\em Invent. Math.}, 216(2):445--518, 2019.
\newblock \href{https://dx.doi.org/10.1007/s00222-018-00848-z}{{\ttfamily
  10.1007/s00222-018-00848-z}}.

\bibitem[Wil18a]{Wilkes2018}
Gareth Wilkes.
\newblock Profinite completions, cohomology and {JSJ} decompositions of compact
  3-manifolds.
\newblock {\em N. Z. J. Math.}, 48:101--113, 2018.

\bibitem[Wil18b]{Wilkes2018b}
Gareth Wilkes.
\newblock Profinite rigidity of graph manifolds and {JSJ} decompositions of
  3-manifolds.
\newblock {\em J. Algebra}, 502:538--587, 2018.
\newblock \href{https://dx.doi.org/10.1016/j.jalgebra.2017.12.039}{{\ttfamily
  10.1016/j.jalgebra.2017.12.039}}.

\bibitem[WZ17]{WiltonZalesskii2017}
Henry Wilton and Pavel Zalesskii.
\newblock Distinguishing geometries using finite quotients.
\newblock {\em Geom. Topol.}, 21(1):345--384, 2017.
\newblock \href{https://dx.doi.org/10.2140/gt.2017.21.345}{{\ttfamily
  10.2140/gt.2017.21.345}}.

\bibitem[WZ19]{WiltonZalesskii2019}
Henry Wilton and Pavel Zalesskii.
\newblock Profinite detection of 3-manifold decompositions.
\newblock {\em Compos. Math.}, 155(2):246--259, 2019.
\newblock \href{https://dx.doi.org/10.1112/S0010437X1800787X}{{\ttfamily
  10.1112/S0010437X1800787X}}.

\bibitem[Wis04]{Wise2004}
D.~T. Wise.
\newblock Cubulating small cancellation groups.
\newblock {\em Geom. Funct. Anal.}, 14(1):150--214, 2004.
\newblock \href{https://dx.doi.org/10.1007/s00039-004-0454-y}{{\ttfamily
  10.1007/s00039-004-0454-y}}.

\bibitem[ZM89]{MelnikovZalesskii1989}
P.~A. Zalesski\u{\i} and O.~V. Mel'nikov.
\newblock Fundamental groups of graphs of profinite groups.
\newblock {\em Algebra i Analiz}, 1(4):117--135, 1989.

\end{thebibliography}
